\documentclass[12pt,letter]{article}
\usepackage[margin=1in]{geometry}
\usepackage[colorlinks=true,citecolor=blue,linkcolor=blue,urlcolor=blue]{hyperref}
\usepackage{booktabs}  
\usepackage{amsmath,amssymb,amsfonts,amsthm}
\usepackage{dsfont}
\usepackage{graphicx}
\graphicspath{{figure_p1/}{./}}
\usepackage{bbm}
\usepackage{xcolor}
\usepackage{algorithm}
\usepackage{algpseudocode}
\graphicspath{{./},{./figures/}}
\usepackage{tikz}
\usepackage{matlab-prettifier}
\usepackage[mathscr]{euscript}
\usepackage{url}

\newtheorem{theorem}{Theorem}
\newtheorem{corollary}{Corollary}
\newtheorem{lemma}{Lemma}
\newtheorem{proposition}{Proposition}

\newtheorem{problem}{Problem}
\newtheorem{remark}{Remark}
\newtheorem{definition}{Definition}
\newtheorem{assumption}{Assumption}

\newcommand{\Ee}{{\rm e}}

\newcommand{\trace}{{\rm tr}}

\newcommand{\vsp}{\vspace{0.07in}}

\newcommand{\calE}{{\mathcal E}}    
\newcommand{\calF}{{\mathcal F}}    
\newcommand{\calG}{{\mathcal G}}    
\newcommand{\calH}{{\mathcal H}}

\newcommand{\calK}{{\mathcal K}}

\newcommand{\calN}{{\mathcal N}}    
    
\newcommand{\calP}{{\mathcal P}}

\newcommand{\calS}{{\mathcal S}}    
\newcommand{\calT}{{\mathcal T}}

\newcommand{\bbE}{{\mathbb E}}

\newcommand{\bbP}{{\mathbb P}}    
    
\newcommand{\bbR}{{\mathbb R}}    
\newcommand{\bbS}{{\mathbb S}}

\newcommand{\rmd}{{\rm d}}

\newcommand{\what}[1]{\widehat{#1}}

\newcommand{\defi}{\mathrel{:=}}

\newenvironment{keywords}
{\par\noindent\textbf{Keywords:}\ }
{\par}

\begin{document}

\title{\resizebox{\textwidth}{!}{Distribution Steering via Sliced Optimal Transport Control}}

\author{
Kaito Ito\\
\small Department of Information Physics and Computing\\
\small The University of Tokyo, Tokyo 113-8654, Japan\\
\small \texttt{kaito@g.ecc.u-tokyo.ac.jp}
\and
Anqi Dong\\
\small Department of Decision and Control Systems\\
\small Department of Mathematics and Digital Futures\\
\small KTH Royal Institute of Technology, 100 44 Stockholm, Sweden\\
\small \texttt{anqid@kth.se}
}

\date{}

\maketitle
\thispagestyle{empty}

\begin{abstract}
Distribution steering seeks feedback laws that drive the state law of a dynamical system between prescribed initial and terminal distributions. Optimal transport provides a natural geometric approach, but its implementation generally requires a transport map or coupling in the full state space. Sliced optimal transport avoids this full-dimensional construction through one-dimensional projections. Yet, the resulting projected maps specify only directional displacements and do not by themselves prescribe a realizable feedback law. To this end, we develop a finite-horizon control framework based on sliced optimal transport. At each sampling instant, a projected optimal transport map defines a directional terminal condition, whose minimum-energy realization yields a randomized single-direction controller. Averaging over projection directions gives a deterministic sliced feedback. For the single-integrator dynamics, the averaged feedback makes the sliced Wasserstein distance to the target non-increasing. For Gaussian endpoint laws, it is affine, preserves Gaussianity, and steers the mean and covariance to their prescribed terminal values. We further identify a law-dependent gain that yields linear decay of the sliced Wasserstein distance together with an explicit characterization of the control energy. We also prove that the randomized controller converges to the averaged sliced flow as the sampling period vanishes. Finally, we extend the construction to linear dynamical systems. Reachability-normalized coordinates allow instantaneous realization of the sliced velocity for uniformly fully actuated systems, while local controllability Gramians provide exact finite-step realization for general controllable systems. Numerical examples illustrate the resulting distributional flows.
\end{abstract}

\begin{keywords}
Optimal transport, sliced Wasserstein distance, distribution steering,
feedback control, linear dynamical systems
\end{keywords}

\section{Introduction}\label{sec:introduction}

Distribution steering seeks to drive the probability law of a dynamical system between prescribed distributions through suitable control inputs. Optimal transport (OT) provides a natural geometric approach to this problem, with origins tracing back to Monge \cite{Monge1781memoire} and the subsequent relaxation of Kantorovich \cite{Kantorovich1942translocation}. Optimal transport has since developed into a broad theory for comparing, interpolating, and transforming probability distributions \cite{Rachev1998mass,Villani2003,villani2009optimal,Peyre2019computational}. A distinguishing feature of OT is that it accounts explicitly for displacement in the underlying state space and therefore retains geometric information that is absent from pointwise comparisons of probability densities. For quadratic transportation cost, the Wasserstein distance admits the dynamic formulation of Benamou and Brenier~\cite{Benamou2000,dong2024monge}, which identifies transport with the minimum kinetic action required to move a density between prescribed endpoint marginals. In this formulation, the density evolves according to the continuity equation and the velocity field plays the role of a control input. This interpretation has motivated a broad line of work on distribution steering, including Schr\"odinger bridge formulations, covariance steering, Gaussian-mixture steering, and optimal-transport-based predictive control \cite{dong2024monge,stephanovitch2025optimal,chen2015optimal1,chen2015optimal2,Chen2018optimal,Ito2023entropic,Ito2023lcss,ito2024maximum}. For controllable linear dynamical systems, distribution steering can further be related to quadratic optimal transport through a transformation involving the finite-horizon reachability Gramian~\cite{Chen2017optimal}.

A common feature of these formulations is their reliance on a transport map or coupling in the full state space. Such an object specifies how mass distributed according to the initial law is reassigned to the target law and thereby determines the displacement to be realized by the controller. This construction is natural from the viewpoint of classical optimal transport. It may, however, become demanding when distributions are represented by large sample sets or when the transport object must be recomputed repeatedly within a feedback loop. Entropic regularization and Sinkhorn-type methods have improved computational scalability, while stochastic and sample-based formulations provide further alternatives \cite{Peyre2019computational,benamou2015iterative,dong2025data}. Nevertheless, these approaches still require constructing, approximating, or sampling from a transport object in the ambient space.

Sliced OT offers a different route by replacing the ambient transport problem with a family of one-dimensional problems along linear projections. Since one-dimensional OT admits explicit solutions through cumulative distribution functions and quantiles, the corresponding transport maps and discrepancies can be efficiently evaluated from samples. This construction is closely related to the Radon-transform viewpoint, where a multidimensional object is represented by its collection of linear projections \cite{Deans2007radon}. Sliced Wasserstein distances have been used for barycenter computation and statistical estimation~\cite{bonneel2015sliced,cuturi2014fast}, as well as for image registration, texture synthesis, and color transfer~\cite{Pitie2007automated,Liutkus2019,solomon2015convolutional}. More recently, they have also played an important role in generative modeling~\cite{nguyen2023sliced}.

Sliced-Wasserstein flows further show that projected optimal transport maps can generate continuous distributional evolutions by averaging projected displacements \cite{Liutkus2019,Cozzi2025}. However, a projected transport map specifies only a directional displacement and does not prescribe how this displacement should be realized by a dynamical system over a prescribed finite horizon. This leaves a gap between projection-based transport and finite-horizon distribution steering. To this end, we develop a finite-horizon control framework based on sliced optimal transport by interpreting projected transport maps as directional terminal conditions. At each sampling instant, a one-dimensional transport map specifies a terminal value along a direction, and the minimum-energy control that reaches this value is lifted to the ambient state space. Repeating this construction with randomly selected directions yields a randomized sliced controller, while its spherical average defines the feedback law analyzed in this paper.

We first study the averaged sliced feedback for the single-integrator dynamics. Proposition~\ref{prop:time_derivative} derives the variation of the sliced Wasserstein distance along the continuity equation and identifies the associated discrepancy field. For Gaussian endpoint laws, Proposition~\ref{prop:linearlity_Gaussian} shows that the feedback preserves Gaussianity and reduces the dynamics to the evolution of the mean and covariance. Theorem~\ref{thm:convergence} proves terminal convergence, while Theorem~\ref{thm:sw_gain} provides a law-dependent gain with linear sliced-Wasserstein decay and explicit control-energy bounds. We then consider the implementable randomized sliced controller obtained by sampling projection directions. Theorem~\ref{thm:main} shows that its sampling-time distributions converge to the averaged sliced flow as the sampling period vanishes, establishing the connection between the discrete controller and the continuous feedback. Finally, we extend the construction to linear dynamical systems. For uniformly fully actuated systems, reachability-normalized coordinates allow the sliced velocity field to be realized instantaneously. For general controllable systems, finite-step displacement matching based on local controllability Gramians reproduces the virtual sliced iteration exactly at the sampling instants, as shown in Proposition~\ref{prop:finite_step_exact}.

The conference version of this work \cite{Ito2026sliced} considered Gaussian marginals for the single-integrator dynamics. Herein, we extend the result by treating weak solutions of the continuity equation under weaker regularity assumptions, proving convergence of the randomized controller, and developing realizations for general linear systems.

Projection-based iterative distribution transfer was introduced in \cite{Pitie2007automated} and later connected with sliced-Wasserstein gradient dynamics in \cite{bonnotte2013unidimensional}. These works show that one-dimensional optimal transport along projections can generate distributional evolutions without constructing a full-dimensional transport map. Sliced-Wasserstein flows further developed this idea by studying continuous-time evolutions generated by averaged projected displacements, including their long-time behavior \cite{Cozzi2025} and critical-point structure \cite{Vauthier2025}.

More recently, stochastic slice matching has been studied as a randomized projection-based transport procedure. In \cite{Li2023measure}, almost-sure convergence to the target distribution is established, while \cite{Thurin2026} derives nonasymptotic convergence rates for Gaussian targets using random orthonormal-basis updates. In contrast, we develop a finite-horizon control interpretation of sliced transport, where projected optimal transport maps define directional terminal conditions, and their minimum-energy realizations generate feedback laws for distribution steering. Theorem~\ref{thm:main} establishes convergence of the randomized controller to the averaged sliced feedback, linking sampled sliced updates with continuous-time distribution steering.

The remainder of the paper is organized as follows. Section~\ref{sec:prelim} reviews quadratic and sliced optimal transport and formulates the distribution-steering problem. Section~\ref{sec:single_integrator} develops the sliced controller and its averaged feedback for the single integrator, including the Gaussian case, terminal convergence, and energy characterization. Section~\ref{sec:convergence} proves convergence of the randomized controller to the averaged sliced flow. Section~\ref{sec:full_actuated} extends the construction to uniformly fully actuated linear systems through reachability-normalized coordinates, while Section~\ref{sec:general_controllable} develops the finite-step realization for general controllable systems. Numerical examples illustrate the resulting distributional evolutions.

Short proofs are included in the main text when they clarify the development, while longer arguments are deferred to the appendices so as not to interrupt the flow of our presentation.

\section{Preliminaries and Problem Formulation}\label{sec:prelim}

\subsection{Optimal mass transport}
Let $\calP_2(\bbR^n)$ denote the space of Borel probability measures on $\bbR^n$ with finite second moment. For an absolutely continuous measure, we slightly abuse notation by using the same symbol for the measure and its density when no confusion arises.
For $\mu,\nu\in\calP_2(\bbR^n)$, the Wasserstein-2 distance is defined by
\begin{equation}\label{def:wasserstein}
W_2^2(\mu,\nu):=
\inf_{\gamma\in\Gamma(\mu,\nu)}
\int_{\bbR^n\times\bbR^n}
\|x-y\|^2\,\gamma(\rmd x,\rmd y),
\end{equation}
where $\Gamma(\mu,\nu) :=
\left\{
\gamma
\ \middle|\
\gamma(\,\cdot\,\times\bbR^n)=\mu,
\gamma(\bbR^n\times\,\cdot\,)=\nu
\right\}$.
A coupling specifies how mass distributed according to $\mu$ is transported to $\nu$, while the objective in \eqref{def:wasserstein} records the expected squared displacement \cite{Rachev1998mass,Villani2003}. In one dimension, the Wasserstein distance admits an explicit quantile representation. Let $F_\mu$ and $F_\nu$ denote the cumulative distribution functions of $\mu,\nu\in\calP_2(\bbR)$, and let 
$$F_\mu^{-1}(z)=\inf\{s\in\bbR\mid F_\mu(s)\geq z\}$$
with $z\in(0,1)$, denote the generalized inverse. Then,
\begin{equation}\label{eq:wasserstein_1d}
W_2^2(\mu,\nu)=\int_0^1 \left|F_\mu^{-1}(z)-F_\nu^{-1}(z)\right|^2\rmd z.
\end{equation}
When $\mu$ is atomless, the optimal transport map is the monotone rearrangement $$T_{\mu\to\nu}(s)=F_\nu^{-1}(F_\mu(s)).$$

Problem in \eqref{def:wasserstein} also admits the Benamou--Brenier (dynamical) formulation \cite{Benamou2000}. In density form, it reads
\begin{equation*}
\inf_{\substack{\rho,v\\
\partial_t\rho+\nabla_x\cdot(\rho v)=0,\\
\rho(\cdot,0)=\mu,\ 
\rho(\cdot,1)=\nu}}
\int_0^1\int_{\bbR^n}
\frac12\,\rho(x,t)\|v(x,t)\|^2
\,\rmd x\,\rmd t .
\end{equation*}
The continuity equation describes conservation of mass, while the objective measures the kinetic energy of the density flow.

Equivalently, following individual trajectories gives the Lagrangian formulation
\begin{equation*}
\inf_{\substack{u:\ \dot x=u\\
x(0)\sim\mu,\ x(1)\sim\nu}}
\int_0^1 \mathbb E \! \left[\frac12\|u(t)\|^2\right]\,\rmd t.
\end{equation*}
The Eulerian and Lagrangian formulations describe the same transport process from the density and particle viewpoints, respectively. For linear dynamics $\dot x(t)=A(t)x(t)+B(t)u(t)$, the same viewpoint leads to an optimal transport problem with prior dynamics~\cite{chen2015optimal1,Chen2017optimal,Chen2021optimal}.

\subsection{Sliced Wasserstein distance}

We next recall the projection operator and the associated sliced Wasserstein distance, see also~\cite{Peyre2019computational,bonneel2015sliced, bonnotte2013unidimensional,paulin2020sliced}.

\begin{definition}[\bf Projection]
For $\theta\in\calS^{n-1}:=\{\theta\in\bbR^n\mid\|\theta\|=1\},$ let $P_\theta(x) \defi \theta^\top x$ denote the projection onto the direction $\theta$. For a map $T$, $T_\# \mu$ denotes the pushforward of $\mu\in \calP_2(\bbR^n)$, and the projected measure $\mu_\theta:=(P_\theta)_\#\mu$ is the distribution of $\theta^\top X$ when $X \sim \mu$.
\end{definition}

\begin{definition}[\bf Sliced Wasserstein distance]
Let $\mu,\nu\in\calP_2(\bbR^n)$, and let $\sigma$ denote the uniform probability measure on $\calS^{n-1}$. The sliced Wasserstein-2 distance is defined by
\[
SW_2^2(\mu,\nu):=\int_{\calS^{n-1}}W_2^2(\mu_\theta,\nu_\theta)\,\sigma(\rmd\theta).
\]
\end{definition}

The sliced Wasserstein distance averages the one-dimensional Wasserstein discrepancies of the projected measures. Using the one-dimensional representation in \eqref{eq:wasserstein_1d}, it can equivalently be written as
\begin{equation*}
SW_2^2(\mu,\nu)=\int_{\calS^{n-1}}\int_0^1
\left|
F_{\mu_\theta}^{-1}(z)
-
F_{\nu_\theta}^{-1}(z)
\right|^2
\rmd z\,
\sigma(\rmd\theta).
\end{equation*}
Thus, $SW_2$ reduces the ambient transport problem to a family of one-dimensional transport problems along linear projections.

\subsection{Problem formulation}

Consider the controlled linear dynamics
\begin{equation}\label{eq:dynamics_general}
\dot{x}(t)=A(t)x(t)+B(t)u(t),
\end{equation}
where $A(t)\in\bbR^{n\times n}$, $B(t)\in\bbR^{n\times m}$, and $u(t)\in\bbR^m$. For $t\in[0,1]$, let $\rho_t$ denote the probability law of $x(t)$, with initial condition $x(0)\sim\rho_0$. A law-dependent feedback is called admissible if the corresponding closed-loop dynamics admit a solution on $t \in [0,1)$ and $\rho_t\in\calP_2(\bbR^n)$. We write
\[
\calP_{2,\mathrm{ac}}(\bbR^n):=
\left\{
\rho\in\calP_2(\bbR^n)
\ \middle|\
\rho\ll\mu_L^n
\right\},
\]
where $\mu_L^n$ denotes the Lebesgue measure on $\bbR^n$.

\begin{problem}\label{prob:main}
Given an initial law $\rho_0\in\calP_{2,\mathrm{ac}}(\bbR^n)$ and a target law $\rho_1\in\calP_2(\bbR^n)$, construct an admissible law-dependent feedback for \eqref{eq:dynamics_general} that steers the state law from $\rho_0$ to $\rho_1$ using only one-dimensional optimal transport maps between projections of the current and target laws.
\end{problem}

The key feature of the proposed approach is that the feedback law is built from one-dimensional projected transport maps, without constructing a transport map or coupling in the ambient space $\bbR^n$. The sliced feedback and its convergence analysis are developed for a single integrator in Sections~\ref{sec:single_integrator} and~\ref{sec:convergence}. Its continuous and finite-step realizations for linear systems are developed in Sections~\ref{sec:full_actuated} and
\ref{sec:general_controllable}, respectively.

\section{Sliced OT Control for the Single Integrator}\label{sec:single_integrator}

We begin with the single-integrator dynamics 
\begin{equation*}
\dot{x}(t)=u(t)
\end{equation*}
where $x(t), u(t)\in\bbR^n$, and the initial state has density $\rho_0$. Under a feedback law $u(t)=v(t,x(t))$, let $\rho(t,\cdot)$ denote the density of $x(t)$. We assume that $\rho$ satisfies the continuity equation in the weak sense, i.e.,
\begin{equation}\label{eq:continuity_eq}
\partial_t\rho(t,x)=-\nabla_x\cdot\left(\rho(t,x)v(t,x)\right),
\end{equation}
on $(0,1)\times\bbR^n$, with initial condition $\rho(0,\cdot)=\rho_0$. More precisely, for every compactly supported smooth test function $\zeta\in C_c^\infty((0,1)\times\bbR^n)$,
\begin{equation}\label{eq:dist_continuity}
\int_0^1\!\!\!\int_{\bbR^n}\!\!\!\left(\partial_t\zeta(t,x)+\nabla_x\zeta(t,x)^\top v(t,x)\right)\rho(t,x)\,\rmd x\,\rmd t=0.
\end{equation}

Given a target density $\rho_1$, the objective is to construct a velocity field $v$ that steers $\rho_0$ to $\rho_1$. For comparison, the classical minimum-energy solution is determined by the optimal transport map $\calT:\bbR^n\to\bbR^n$ from $\rho_0$ to $\rho_1$~\cite{Chen2017optimal}. Let $\calT_{0:t}(x)\defi(1-t)x+t\calT(x)$ denote the displacement interpolation. The associated velocity field is $v(t,x)=\calT\circ\calT_{0:t}^{-1}(x)-\calT_{0:t}^{-1}(x)$. Moreover, $W_2^2(\rho_0,\rho_1)=\int_{\bbR^n}\|x-\calT(x)\|^2\rho_0(x)\,\rmd x$. Thus, the classical controller requires an optimal transport map in the full state space. Constructing such a map between continuous distributions is generally difficult in high dimensions. This motivates the sliced formulation developed next.

\subsection{Iterative sliced controller and averaged feedback}\label{subsec:sliced_feedback}

For $\theta\in\calS^{n-1}$, define $\rho_{t,\theta}:=(P_\theta)_\#\rho_t$, and let $\calT_t^\theta:\bbR\to\bbR$ denote the monotone optimal transport map from $\rho_{t,\theta}$ to $\rho_{1,\theta}$. Since $\rho_t$ admits a density, so does $\rho_{t,\theta}$, and $\calT_t^\theta$ is uniquely defined $\rho_{t,\theta}$-almost everywhere. Moreover,
\begin{equation*}
W_2^2\bigl(\rho_{t,\theta},\rho_{1,\theta}\bigr)=\int_{\bbR}\left|s-\calT_t^\theta(s)\right|^2\rho_{t,\theta}(s)\,\rmd s.
\end{equation*}

We first construct a controller that uses one projection direction at each sampling instant. Let $t_k\defi kh$, for $k\in\{0,1,\ldots,N\}$, where $h>0$ and $N\defi 1/h$ is assumed to be an integer. The control is held constant over each interval $[t_k,t_{k+1})$, so that the single-integrator dynamics becomes
$$x_{k+1}=x_k+h u_k.$$

At time $t_k$, select a direction $\theta_k\in\calS^{n-1}$ and let $\calT_k^{\theta_k}:=\calT_{t_k}^{\theta_k}$. The projected state $s_k\defi \theta_k^\top x_k$ evolves according to
\begin{equation*}
s_{\ell+1}=s_\ell+h u_{\theta_k,\ell},
\end{equation*}
where $u_{\theta_k,\ell}$ is the scalar control along $\theta_k$.  To steer $s_k$ to its transported image $\calT_k^{\theta_k}(s_k)$ at the terminal time, consider
\begin{equation}
\min_{\{u_{\theta_k,\ell}\}}\sum_{\ell=k}^{N-1}
h\,u_{\theta_k,\ell}^2 \quad \text{\rm s.t.}\quad s_{\ell+1}=s_\ell+h u_{\theta_k,\ell}
\end{equation}
with $s_k=\theta_k^\top x_k$ and $s_N=\calT_k^{\theta_k}(\theta_k^\top x_k)$. The minimum-energy control is constant over the remaining time steps $\ell\in\{k,\ldots,N-1\}$ and is given by
\begin{equation*}
u_{\theta_k,\ell}^*=-\frac{1}{1-t_k}\! \left(
\theta_k^\top x_k
-
\calT_k^{\theta_k}(\theta_k^\top x_k)
\right).
\end{equation*}

In a receding-horizon implementation, only the first control value is applied. Lifting it from the projected coordinate to $\bbR^n$ gives
\begin{equation}\label{eq:slice_controller}
u_k=-\frac{\theta_k^\top x_k-\calT_k^{\theta_k}(\theta_k^\top x_k)}{1-t_k}
\theta_k.
\end{equation}
At the next sampling instant, the projection direction and the one-dimensional transport map are updated. Repeating this procedure gives the iterative sliced controller.

For the analysis below, we allow a general nonnegative gain $\lambda:[0,1)\to[0,\infty)$ and write the directional controller as
\begin{equation}\label{eq:iterative_sliced_controller}
u_k=-\lambda(t_k)\! \left(\theta_k^\top x_k-\calT_k^{\theta_k}(\theta_k^\top x_k)\right)\theta_k.
\end{equation}
The choice $\lambda(t)=(1-t)^{-1}$ recovers \eqref{eq:slice_controller}. More general choices will be useful in the analysis of the averaged flow and its control energy.

Averaging the directional update in \eqref{eq:iterative_sliced_controller} over $\theta\sim \sigma$ gives the velocity field
\begin{equation}\label{eq:optimal_controller_single_integ}
v(t,x)=-\lambda(t)g_t(x),
\end{equation}
where
\begin{equation}\label{eq:sliced_discrepancy}
g_t(x)
:=
\int_{\calS^{n-1}}
\left(
\theta^\top x
-
\calT_t^\theta(\theta^\top x)
\right)\theta
\,\sigma(\rmd\theta).
\end{equation}
We refer to $g_t$ as the sliced discrepancy field and to \eqref{eq:optimal_controller_single_integ} as the averaged sliced
feedback. The iterative controller \eqref{eq:iterative_sliced_controller} is its randomized single-direction realization.

\begin{remark}
The gain in \eqref{eq:slice_controller} is the reciprocal of the remaining time and therefore increases as $t_k$ approaches the
terminal time. Alternatively, the projected endpoint problem may be solved over a fixed receding horizon $\tau>0$. The corresponding control is
\begin{equation*}
u_k
=
-\frac{1}{\tau}
\! \left(
\theta_k^\top x_k
-
\calT_k^{\theta_k}(\theta_k^\top x_k)
\right)\theta_k,
\end{equation*}
and the state update becomes
\begin{equation*}
x_{k+1}
=
x_k
-
\frac{h}{\tau}
\! \left(
\theta_k^\top x_k
-
\calT_k^{\theta_k}(\theta_k^\top x_k)
\right)\theta_k.
\end{equation*}
Thus, $h/\tau$ acts as a step size of a directional descent iteration. Applying several mutually orthogonal directions in
succession gives a procedure closely related to iterative distribution transfer \cite{Pitie2007automated}. Related combinations of optimal transport and model predictive control for discrete distributions appear in \cite{Ito2023entropic,Ito2023lcss}.
\end{remark}

\subsection{Averaged sliced flow and descent of the sliced Wasserstein distance}\label{subsec:descent}

The averaged sliced feedback is related to the variation of the sliced Wasserstein distance along the continuity equation. In this subsection, let $\rho_t(\rmd x)=\rho(t,x)\rmd x$ be a narrowly continuous curve in $\calP_2(\bbR^n)$ satisfying \eqref{eq:continuity_eq} weakly. We assume that the kinetic energy of $v$ is finite on every compact subinterval of $[0,1)$, namely, 
$$\int_0^{\bar t}\int_{\bbR^n}\|v(t,x)\|^2\,\rho_t(\rmd x)\,\rmd t<\infty$$ 
for $\bar{t}\in (0,1)$.

\begin{proposition}
\label{prop:time_derivative}
The function $r(t):=SW_2(\rho_t,\rho_1)$ is locally absolutely continuous on $[0,1)$. Moreover, for almost every $t\in(0,1)$, we have $\frac{\rmd}{\rmd t}SW_2^2(\rho_t,\rho_1)=2\int_{\bbR^n}g_t(x)^\top v(t,x)\,\rho_t(\rmd x)$.
\end{proposition}

\vsp
\begin{proof}
See Appendix~\ref{app:time_derivative}.
\end{proof}

For averaged sliced feedback $v(t,x)=-\lambda(t)g_t(x)$, define
\begin{equation}\label{eq:D-def}
D(t):=\int_{\bbR^n}\|g_t(x)\|^2\,\rho_t(\rmd x).
\end{equation}
Proposition~\ref{prop:time_derivative} then gives
\begin{equation}\label{eq:dissipation_general}
\frac{\rmd}{\rmd t}
SW_2^2(\rho_t,\rho_1)=-2\lambda(t)D(t)\leq 0
\end{equation}
for almost every $t\in(0,1)$. Hence, the sliced Wasserstein distance is non-increasing along the closed-loop flow.

\begin{remark}[\bf Directional consistency]\label{rmk:chi}
Let $a_t^\theta(x):=\theta^\top x-\calT_t^\theta(\theta^\top x)$ denote the sliced displacement at $x$ along $\theta$. Then
\begin{align*}
r(t)^2&=\int_{\bbR^n}\int_{\calS^{n-1}}|a_t^\theta(x)|^2\,\sigma(\rmd\theta)\,\rho_t(\rmd x),\\
D(t)&=\int_{\bbR^n}\Big\|\int_{\calS^{n-1}}a_t^\theta(x)\theta\,\sigma(\rmd\theta)\Big\|^2\rho_t(\rmd x).
\end{align*}
Thus, $r(t)^2$ measures the total sliced discrepancy, whereas $D(t)$ measures the part retained after the lifted directional displacements are averaged over the sphere. By the Cauchy--Schwarz inequality and $\int_{\calS^{n-1}}\theta\theta^\top\sigma(\rmd\theta) = I/n$,
\begin{equation}\label{eq:chi-upper}
D(t)\leq\frac{1}{n}r(t)^2.
\end{equation}
For $r(t)>0$, define
\begin{equation}\label{eq:r-chi-def}
\chi(t)
:=
\frac{n}{r(t)^2}D(t),
\end{equation}
and set $\chi(t)=1$ when $r(t)=0$. Then $0\leq\chi(t)\leq1$. Values of $\chi(t)$ near one indicate that the directional displacements reinforce one another, while small values indicate substantial cancellation.
\end{remark}

\subsection{Gaussian marginals and affine feedback}\label{subsec:gaussian_affine}
We now take $\rho_i=\calN(\,\cdot\mid m_i,\Sigma_i)$, $i=0,1$, where $\calN(\,\cdot\mid m,\Sigma)$ denotes the Gaussian density with mean $m$ and covariance $\Sigma$, and $\Sigma_i\succ0$. Since linear projections preserve Gaussianity, $\rho_{i,\theta}= \calN(\,\cdot\mid\theta^\top m_i,\theta^\top\Sigma_i\theta)$ for $i=0,1$. For $\Sigma\succ0$, define
\begin{equation*}
\alpha(\Sigma,\theta)
:=
\sqrt{
\frac{\theta^\top\Sigma_1\theta}
{\theta^\top\Sigma\theta}
}.
\end{equation*}
The optimal transport map between one-dimensional Gaussian distributions
is affine \cite[Remark~2.31]{Peyre2019computational}. Hence, the map from
$\rho_{0,\theta}$ to $\rho_{1,\theta}$ is
\begin{equation*}
\calT_0^\theta(\theta^\top x)=\theta^\top m_1+\alpha(\Sigma_0,\theta)\left(\theta^\top x-\theta^\top m_0\right).
\end{equation*}

Introduce the matrix-valued functions
\begin{align*}
H(\Sigma)
&:=
\int_{\calS^{n-1}}
\alpha(\Sigma,\theta)\theta\theta^\top
\,\sigma(\rmd\theta)
-
\frac{1}{n}I,\\
K(t,\Sigma)
&:=
\lambda(t)H(\Sigma).
\end{align*}
Using the spherical identity $\int_{\calS^{n-1}}\theta\theta^\top\sigma(\rmd\theta)=I/n$, the averaged sliced feedback at the initial time becomes
\begin{equation*}
\begin{aligned}
v(0,x)=K(0,\Sigma_0)x-\lambda(0) \Big[\big(H(\Sigma_0)+\frac{1}{n}I\big)m_0-\frac{1}{n}m_1\Big].
\end{aligned}
\end{equation*}
Thus, the initial feedback is affine in the state. The following proposition shows that this affine structure is preserved by the averaged sliced flow.

\begin{proposition}[\bf Gaussian preservation and affine feedback]
\label{prop:linearlity_Gaussian}
For continuous $\lambda:[0,1)\to[0,\infty)$, the covariance equation
\begin{equation}\label{eq:cov_eq}
\dot{\Sigma}(t)=K(t,\Sigma(t))\Sigma(t)+\Sigma(t)K(t,\Sigma(t))^\top
\end{equation}
with $\Sigma(0)=\Sigma_0$ admits a unique solution satisfying $\Sigma(t)\succ0$ on $[0,1)$. Define
\begin{equation}\label{eq:m_solution}
m(t) \defi m_1+\exp \Big(-\frac{1}{n}
\int_0^t\lambda(s)\,\rmd s\Big)(m_0-m_1)
\end{equation}
and
\begin{equation*}
\eta(t)\defi-\lambda(t)\Big[\big(H(\Sigma(t))+\frac{1}{n}I\big)m(t)-\frac{1}{n}m_1\Big].
\end{equation*}
With initial law $X(0)\sim\calN(m_0,\Sigma_0)$, the affine system
\begin{equation}\label{eq:X_affine}
\dot{X}(t)
=
K(t,\Sigma(t))X(t)+\eta(t)
\end{equation}
admits a unique solution on $[0,1)$ and satisfies $X(t)\sim\calN\bigl(m(t),\Sigma(t)\bigr)$. Moreover, if $g_t$ is computed from $\rho_t=\calN(\,\cdot\mid m(t),\Sigma(t))$ and the target density $\rho_1$, then
\begin{equation*}
-\lambda(t)g_t(x)
=
K(t,\Sigma(t))x+\eta(t).
\end{equation*}
Hence, the averaged sliced feedback is affine and preserves Gaussianity.
\end{proposition}

\begin{proof}
See Appendix~\ref{app:linearity_Gaussian}.
\end{proof}

\subsection{Terminal convergence and energy characterization}
\label{subsec:terminal_energy}

Proposition~\ref{prop:linearlity_Gaussian} shows that the averaged sliced feedback preserves Gaussianity. We first establish conditions under which this Gaussian flow reaches the prescribed terminal law. Since a Gaussian law is determined by its mean and covariance, it is enough to study the limits of $m(t)$ and $\Sigma(t)$.

\begin{theorem}[\bf Terminal convergence]\label{thm:convergence}
Denote the Gaussian flow in Proposition~\ref{prop:linearlity_Gaussian} as $\rho_t=\calN(\,\cdot\mid m(t),\Sigma(t))$. Suppose that
\begin{equation}\label{eq:cumulative-gain}
\int_0^1\lambda(t)\,\rmd t=\infty
\end{equation}
and that $\Sigma(t)\succeq\alpha I$ on $[0,1)$ for some $\alpha>0$. Set
\begin{equation*}
\beta:=n \Big(\sqrt{\frac{\trace(\Sigma_1)}{n}}+r(0)\Big)^2 \text{\rm and } \chi_*:=\min\! \big\{1,\frac{2\alpha\sqrt{\lambda_{\min}(\Sigma_1)}}{\beta^{3/2}(n+2)}\big\},
\end{equation*}
then we have 
\begin{equation}\label{eq:r-decay-general}
r(t)\leq r(0)\exp\! \Big(-\frac{\chi_*}{n}\int_0^t\lambda(s)\,\rmd s\Big).
\end{equation}
Consequently, $m(t)\to m_1$ and $\Sigma(t)\to\Sigma_1$ as $t\nearrow1$.
\end{theorem}

\begin{proof}
See Appendix~\ref{app:convergence}.
\end{proof}

The same affine feedback also governs the first two moments of any law with finite second moment. Thus, under the conditions of Theorem~\ref{thm:convergence}, the feedback law $u(t,x)=K(t,\Sigma(t))x+\eta(t)$ steers the mean and covariance from $(m_0,\Sigma_0)$ to $(m_1,\Sigma_1)$, even when the state law is not Gaussian.

We next characterize the control effort. The following results apply to general averaged sliced flows and do not require Gaussianity. Let $\rho_t$ be generated by $v(t,x)=-\lambda(t)g_t(x)$ and satisfy the assumptions of Proposition~\ref{prop:time_derivative} on every compact subinterval of $[0,1)$. For $X(t)\sim\rho_t$, write $u(t)=v(t,X(t))$ and, for $0<\tau<1$, define
\begin{equation*}
\calE_u(0,\tau)
:=
\int_0^\tau
\bbE\! \left[\|u(t)\|^2\right]
\,\rmd t.
\end{equation*}

\begin{lemma}[\bf Control energy along the sliced flow]\label{lem:energy_characterization}
Suppose that $D(t)>0$ whenever $r(t)>0$. Then, for almost every
$t\in(0,1)$,
\begin{equation}\label{eq:power-chi}
\bbE\! \left[\|u(t)\|^2\right]
=
\frac{n\dot r(t)^2}{\chi(t)}.
\end{equation}
Consequently,
\begin{equation}\label{eq:energy-chi}
\calE_u(0,\tau)
=
n\int_0^\tau
\frac{\dot r(t)^2}{\chi(t)}
\,\rmd t.
\end{equation}
If $r(t)\to0$ as $t\nearrow1$, then
\begin{equation}\label{eq:energy-lower}
\calE_u(0,1)
\geq
nSW_2^2(\rho_0,\rho_1),
\end{equation}
where
$\calE_u(0,1):=\lim_{\tau\nearrow1}\calE_u(0,\tau)$.
\end{lemma}

\begin{proof}
See Appendix~\ref{app:energy_lem}.
\end{proof}

Lemma~\ref{lem:energy_characterization} separates the rate at which the sliced discrepancy decreases from the loss caused by averaging over projection directions. For a fixed decay rate, a smaller value of $\chi(t)$ requires greater control effort.

Combining the dissipation identity in \eqref{eq:dissipation_general} with \eqref{eq:r-chi-def} gives $$\dot r(t)=-\frac{\lambda(t)\chi(t)}{n}r(t)$$ whenever $r(t)>0$. To obtain linear decay over the remaining horizon, define
\begin{equation}\label{eq:lambda-sw}
\lambda_{\rm SW}(t)
:=
\frac{r(t)^2}{D(t)(1-t)}
=
\frac{n}{\chi(t)(1-t)},
\qquad r(t)>0 .
\end{equation}
Here, $(1-t)^{-1}$ accounts for the remaining time, while $n/\chi(t)$ compensates for directional cancellation.

\begin{theorem}[\bf Linear sliced decay and energy bounds]\label{thm:sw_gain}
Suppose that the assumptions of Lemma~\ref{lem:energy_characterization} hold for $\lambda(t)=\lambda_{\rm SW}(t)$ and that $r(0)>0$. Then, for every $t\in[0,1)$,
\begin{equation}\label{eq:r-linear}
r(t)
=
(1-t)r(0).
\end{equation}
The corresponding total control energy is
\begin{equation}\label{eq:energy-sw}
\calE_u(0,1)
=
nSW_2^2(\rho_0,\rho_1)
\int_0^1
\frac{1}{\chi(t)}
\,\rmd t.
\end{equation}
If $\chi(t)\geq\underline{\chi}>0$ for almost every $t\in(0,1)$, then
\begin{equation}\label{eq:energy-sw-bounds}
nSW_2^2(\rho_0,\rho_1)
\leq
\calE_u(0,1)
\leq
\frac{n}{\underline{\chi}}
SW_2^2(\rho_0,\rho_1).
\end{equation}
\end{theorem}

\begin{proof}
See Appendix~\ref{app:energy}.
\end{proof}

For Gaussian flows, the averaged sliced field cannot vanish before the target law is reached. The gain in \eqref{eq:lambda-sw} is therefore well defined whenever $r(t)>0$.
\begin{corollary}[\bf Gaussian flow]
\label{cor:Gaussian_SW_normalized}
Assume that the mean and covariance equations of Proposition~\ref{prop:linearlity_Gaussian} with $\lambda(t)=\lambda_{\mathrm{SW}}(t)$ admit a solution $(m(t),\Sigma(t)) \in \mathbb{R}^n\times\mathbb{S}_{>0}^n$, $t\in[0,1)$ and that the curve $\rho_t=\mathcal{N}(\cdot\mid m(t),\Sigma(t))$ satisfies the assumptions of Proposition~\ref{prop:time_derivative}. Then $D(t)>0$ whenever $r(t)>0$ and
\begin{align}
\|m(t)-m_1\|
&\leq
\sqrt{n}(1-t)r(0),
\label{eq:m_bound}\\
\|\Sigma(t)-\Sigma_1\|_{\rm F}
&\leq
\sqrt{\frac{2}{\mu}}\,
(1-t)r(0),
\label{eq:Sigma_bound}
\end{align}
where $\mu>0$ is defined in \eqref{eq:mu_def}. Moreover, $\Sigma(t)\succeq\alpha I$ on $[0,1)$ for some $\alpha>0$, and
\begin{equation*}
nSW_2^2(\rho_0,\rho_1)
\leq
\calE_u(0,1)
\leq
\frac{n}{\chi_*}
SW_2^2(\rho_0,\rho_1),
\end{equation*}
where $\chi_*$ is defined as in Theorem~\ref{thm:convergence} using this value of $\alpha$.
\end{corollary}

\begin{proof}
See Appendix~\ref{app:cor_gauss}.
\end{proof}

\section{Convergence of the Iterative Sliced Controller}\label{sec:convergence}

We now show that the iterative sliced controller approaches the averaged sliced flow as the sampling period tends to zero. To isolate the approximation due to random slicing, we restrict attention to the single-integrator dynamics.

For a current distribution $\rho$ and a direction $\theta\in\calS^{n-1}$, write
\begin{align*}
g_\theta[\rho](x)
&:=
\left(
\theta^\top x-\calT_\rho^\theta(\theta^\top x)
\right)\theta,\\
\bar g[\rho](x)
&:=
\int_{\calS^{n-1}}
g_\theta[\rho](x)\,\sigma(\rmd\theta),
\end{align*}
where $\calT_\rho^\theta$ is the one-dimensional optimal transport map from $\rho_\theta$ to $\rho_{1,\theta}$. Thus, $g_t=\bar g[\rho_t]$ when $\rho=\rho_t$.

Fix $\bar t\in(0,1)$ and assume that $\lambda$ is continuously differentiable on $[0,\bar t]$. Set $\Lambda_{\bar t}:=\max_{0\leq t\leq\bar t}\lambda(t)$. Let $x_0\sim \rho_0$ be defined on a probability space $(\Omega_0,\calF_0,\bbP_0)$, write $\bbE_0$ for expectation, and set $\|X\|_{L^2(\Omega_0)}^2 := \bbE_0[\|X\|^2]$. On a separate probability space, let $\{\theta_k\}_{k\geq0}$ be an i.i.d. sequence with common distribution $\sigma$, independent of $x_0$. For a sampling period $h>0$, set $t_k\defi kh$ and $K_h\defi \lfloor\bar t/h\rfloor$.

Let $\Theta_k$ denote the sigma-algebra generated by $\{\theta_0,\ldots,\theta_{k-1}\}$, representing the direction history before time $t_k$. We write $\rho_k^h:=\operatorname{Law}(x_k^h\mid\Theta_k)$ for the conditional distribution of $x_k^h$, leaving its dependence on $\Theta_k$ implicit. The iterative sliced controller generates
\begin{equation}
x_{k+1}^h
=
x_k^h
-
h\lambda(t_k)
g_{\theta_k}[\rho_k^h](x_k^h),
\label{eq:random-scheme}
\end{equation}
with $x_0^h=x_0$. Equivalently, for each realization of the direction sequence,
\begin{equation}
\rho_{k+1}^h
=
\left(
\operatorname{Id}
-
h\lambda(t_k)
g_{\theta_k}[\rho_k^h]
\right)_\#
\rho_k^h,
\label{eq:law-update_iter}
\end{equation}
where $\operatorname{Id}$ denotes the identity map.

For the averaged flow, we make the following assumption.
\begin{assumption}[\bf Existence]
\label{ass:averaged-flow}
There exists a $W_2$-continuous curve $\{\rho_t\}_{t\in [0,\bar{t}]} \subset \calP_{2,{\rm ac}} (\bbR^n) $ and a Borel measurable map $\Phi_t :\bbR^n\to\bbR^n$ such that $\Phi_0=\operatorname{Id}$ and, for every $x\in\bbR^n$, the curve $t\mapsto\Phi_t(x)$ is absolutely continuous on $ [0,\bar{t}] $ and satisfies, for almost every $t\in(0,\bar t)$ and for $\rho_0$-almost every $x$, $$\frac{\rmd}{\rmd t}\Phi_t(x)= -\lambda(t)\bar g[\rho_t](\Phi_t(x))$$ with $\rho_t=(\Phi_t)_\#\rho_0$.
\end{assumption}

Assumption~\ref{ass:averaged-flow} fixes a reference averaged flow in Lagrangian form. It is a well-posedness assumption for the comparison flow.
Set $y(0)=x(0)$ and define $y(t):=\Phi_t(x(0))$. Then, $y(t)$ has law $\rho_t$ and satisfies
\begin{equation}
\frac{\rmd}{\rmd t}y(t)
=
-\lambda(t)\bar g[\rho_t](y(t)).
\label{eq:ideal-flow}
\end{equation}
Our objective is to compare $\rho_k^h$ with $\rho_{t_k}$ as $h\searrow0$.

For $\mu\in\calP_2(\bbR^n)$, write $$m_2(\mu):=(\int_{\bbR^n}\|x\|^2\,\rmd\mu(x))^{1/2}.$$ The moment estimates established below show that the iterative and averaged distributions remain in the class
\begin{equation*}
\calK_{\bar t}
:=
\left\{
\rho\in\calP_{2,\mathrm{ac}}(\bbR^n)
\ \middle|\
m_2(\rho)
\leq
\bigl(m_2(\rho_0)+m_2(\rho_1)\bigr)
\Ee^{\Lambda_{\bar t}\bar t}
\right\}.
\end{equation*}

To compare the two evolutions, we require stability along the conditional laws generated by the iterative sliced controller and along the averaged sliced flow.

\begin{assumption}[\bf Stability along the generated flows]\label{ass:main}
There exist constants $h_{\bar t}\in(0,1]$ and $L_{\bar t}<\infty$ such that $h_{\bar t}\Lambda_{\bar t}<1$ and
\begin{enumerate}
\item for any $h\in(0,h_{\bar t}]$, $k\in\{0,\ldots,K_h\}$, and almost every realization of the direction history,
\begin{equation*}
\left\|
\bar g[\rho_k^{h}](x_k^h)
-
\bar g[\rho_{t_k}](y(t_k))
\right\|_{L^2(\Omega_0)}
\le
L_{\bar t}
\left\|x_k^h-y(t_k)\right\|_{L^2(\Omega_0)}.
\end{equation*}
\item for any $s,t\in[0,\bar t]$,
\begin{equation*}
\left\|
\bar g[\rho_s](y(s))
-
\bar g[\rho_t](y(t))
\right\|_{L^2(\Omega_0)}\le
L_{\bar t}
\left\|y(s)-y(t)\right\|_{L^2(\Omega_0)}.
\end{equation*}
\end{enumerate}
\end{assumption}
\vsp

We are now ready to state the convergence result. The expectation below is taken with respect to the sampled directions.

\begin{theorem}[\bf Vanishing-step convergence]\label{thm:main}
Suppose that Assumptions~\ref{ass:averaged-flow} and \ref{ass:main} hold. Then, there exists a constant $C_{\bar t}<\infty$, independent of $h$, such that, for every $h\in(0,h_{\bar t}]$,
\begin{align}
\bbE_{\{\theta_k\}}
\! \left[
\max_{0\leq k\leq K_h}
W_2^2\! \left(
\rho_k^h,\rho_{t_k}
\right)
\right]
&\leq
C_{\bar t}h,
\label{eq:main-W2}\\
\bbE_{\{\theta_k\}}
\! \left[
\max_{0\leq k\leq K_h}
SW_2^2\! \left(
\rho_k^h,\rho_{t_k}
\right)
\right]
&\leq
\frac{C_{\bar t}}{n}h.
\label{eq:main-SW2}
\end{align}
For a fixed $t\in[0,\bar t]$, let $k_h\defi \lfloor t/h\rfloor$. 
Then, we have
$$\lim_{h\searrow0}\bbE_{\{\theta_k\}}\! \left[W_2^2\! \left(\rho_{k_h}^h,\rho_t\right)\right]=0.$$
\end{theorem}

Thus, the iterative controller converges to the averaged sliced flow with an $O(h)$ expected squared Wasserstein error.  The proof proceeds through uniform moment estimates, a decomposition of each randomized update into its averaged drift and a centered fluctuation, and a discrete stability argument. We establish these ingredients next and return to the proof of Theorem~\ref{thm:main} thereafter.

\subsection{Estimates for the convergence proof}

Throughout this subsection, we retain the notation and assumptions introduced above. In particular, the analysis is carried out on $[0,\bar t]$, where $\lambda\in C^1([0,\bar t])$. We also set
$$
\Lambda'_{\bar t}:=\max_{0\leq t\leq\bar t}|\dot\lambda(t)|.
$$

The randomized sliced iteration can be viewed as a stochastic Euler approximation of the averaged flow. The random projection direction produces a centered fluctuation, while time discretization introduces a local Euler residual. We first record an abstract estimate for these two effects and then show that it applies to the sliced dynamics.

\begin{proposition}[\bf Stochastic Euler estimate]\label{prop:hilbert-euler}
Let $\calH$ be a Hilbert space and $b:[0,\bar t]\times\calH\to\calH$. Suppose that $z_0^h=z(0)\in\calH$ and
\begin{align*}
z_{k+1}^h
&=
z_k^h+h b(t_k,z_k^h)+h\xi_{k+1},\\
z(t_{k+1})
&=
z(t_k)+h b(t_k,z(t_k))+r_k .
\end{align*}
Assume that $\|r_k\|_{\calH}\leq Bh^2$ and, for every $k\in\{0,\ldots,K_h\}$,
\begin{equation}
\|b(t_k,z_k^h)-b(t_k,z(t_k))\|_\calH
\le L\|z_k^h-z(t_k)\|_\calH.
\label{eq:trajectory-stability-euler}
\end{equation}
Moreover, let $\{\xi_{k+1}\}$ be an $\calH$-valued martingale difference sequence with respect to $\{\mathcal F_k\}$ satisfying $\mathbb E[\xi_{k+1}\mid\mathcal F_k]=0$ and $\mathbb E[\|\xi_{k+1}\|_{\calH}^2\mid\mathcal F_k]\leq\sigma_\xi^2$. Then there exists $C<\infty$, independent of $h$, such that
\begin{equation}\label{eq:H-estimate}
\mathbb E\big[
\max_{0\leq k\leq K_h}
\|z_k^h-z(t_k)\|_{\calH}^2
\big] \leq Ch,
\end{equation}
for every $h\in(0,1]$.
\end{proposition}

\begin{proof}
Set $e_k:=z_k^h-z(t_k)$. Since $e_0=0$, summing the error recursion gives
\begin{equation*}
e_j
=
h\sum_{i=0}^{j-1}
\bigl(
b(t_i,z_i^h)-b(t_i,z(t_i))
\bigr)
+
h\sum_{i=0}^{j-1}\xi_{i+1}
-
\sum_{i=0}^{j-1}r_i.
\end{equation*}
For $j\leq k$, the trajectory-wise stability condition \eqref{eq:trajectory-stability-euler} and the Cauchy--Schwarz inequality give
\[
\Big\|
h\sum_{i=0}^{j-1}
\bigl(
b(t_i,z_i^h)-b(t_i,z(t_i))
\bigr)
\Big\|_\calH^2
\leq
\bar tL^2h
\sum_{i=0}^{k-1}
\max_{0\leq r\leq i}\|e_r\|_\calH^2.
\]
Moreover,
$\|\sum_{i=0}^{j-1}r_i\|_\calH\leq B\bar t h$. Hence,
\begin{align}
\max_{0\leq j\leq k}\|e_j\|_\calH^2
\leq
3L^2\bar t h
\sum_{i=0}^{k-1}
\max_{0\leq r\leq i}\|e_r\|_\calH^2
+
3\max_{0\leq j\leq k}
\Big\|
h\sum_{i=0}^{j-1}\xi_{i+1}
\Big\|_\calH^2
+
3B^2\bar t^2h^2.
\label{eq:31-detail}
\end{align}

By Doob's $L^2$ inequality and the martingale-difference property,
\begin{equation}
\bbE\Big[
\max_{0\leq j\leq k}
\big\|
h\sum_{i=0}^{j-1}\xi_{i+1}
\big\|_\calH^2
\Big]
\leq
4\sigma_\xi^2\bar t h.
\label{eq:32-detail}
\end{equation}
Indeed, the cross terms vanish and
\[
\bbE\Big[
\big\|
h\sum_{i=0}^{k-1}\xi_{i+1}
\big\|_\calH^2
\Big]
=
h^2\sum_{i=0}^{k-1}
\bbE\! \left[\|\xi_{i+1}\|_\calH^2\right]
\leq
\sigma_\xi^2\bar t h.
\]

Set $a_k:=\bbE[\max_{0\leq j\leq k}\|e_j\|_\calH^2]$. Taking expectations in \eqref{eq:31-detail} and using
\eqref{eq:32-detail} gives
\[
a_k
\leq
3L^2\bar t h\sum_{i=0}^{k-1}a_i
+
12\sigma_\xi^2\bar t h
+
3B^2\bar t^2h^2.
\]
The discrete Gronwall inequality proves \eqref{eq:H-estimate}. One may take $C=\Ee^{3L^2\bar t^2} (12\sigma_\xi^2\bar t+3B^2\bar t^2)$.
\end{proof}

We next establish the estimates needed to apply Proposition~\ref{prop:hilbert-euler} to the sliced dynamics. For a measurable function $f:\bbR^n\to\bbR^n$, write $\|f\|_{L_\rho^2}^2:=\int_{\bbR^n}\|f(x)\|^2\,\rmd\rho(x)$.

\begin{lemma}[\bf Norm of a sliced displacement]
\label{lem:slice-id}
For every $\rho\in\calP_{2,\mathrm{ac}}(\bbR^n)$ and $\theta\in\calS^{n-1}$, we have
$\|g_\theta[\rho]\|_{L_\rho^2}^2
=
W_2^2(\rho_\theta,\rho_{1,\theta})$.
Consequently,
$\|g_\theta[\rho]\|_{L_\rho^2}
\leq
m_2(\rho)+m_2(\rho_1)$.
\end{lemma}

\begin{proof}
See Appendix~\ref{sec:proof-lemma1}.
\end{proof}

\begin{lemma}[\bf Preservation of absolute continuity]
\label{lem:ac-preservation}
Let $\rho\in\calP_{2,\mathrm{ac}}(\bbR^n)$, $\theta\in\calS^{n-1}$, and $\gamma\in[0,1)$. Then the map
\begin{equation}
T_{\rho,\theta,\gamma}(x)
\defi
x-\gamma g_\theta[\rho](x)
\label{eq:iter_map}
\end{equation}
satisfies
$(T_{\rho,\theta,\gamma})_\#\rho
\in\calP_{2,\mathrm{ac}}(\bbR^n)$.
Consequently, the laws $\rho_k^h$ generated by \eqref{eq:law-update_iter} with $h\in(0,h_{\bar t}]$ are absolutely continuous for every $k\leq K_h$, for almost every realization of the direction sequence.
\end{lemma}

\begin{proof}
See Appendix~\ref{sec:proof-lemma2}.
\end{proof}

\begin{lemma}[\bf Uniform moment bounds]
\label{lem:moment-confinement}
For almost every realization of the direction sequence,
\begin{equation}
m_2(\rho_k^h)
\leq
\bigl(m_2(\rho_0)+m_2(\rho_1)\bigr)
\Ee^{\Lambda_{\bar t}\bar t},
\qquad
0\leq k\leq K_h.
\label{eq:random-moment-bound}
\end{equation}
In addition, suppose that Assumption~\ref{ass:averaged-flow} holds. Then, the same bound holds for $m_2(\rho_t)$ on $[0,\bar t]$. Consequently, the iterative and averaged laws remain in $\calK_{\bar t}$.
\end{lemma}

\begin{proof}
See Appendix~\ref{sec:proof-lemma3}.
\end{proof}

The preceding estimates also give uniform control of the directional and averaged sliced fields. Set $M_{\bar t}:=
\bigl(m_2(\rho_0)+m_2(\rho_1)\bigr) \Ee^{\Lambda_{\bar t}\bar t}+m_2(\rho_1)$.

\begin{lemma}[\bf Uniform bounds for the sliced fields]
\label{lem:MS-bound}
For every $\rho\in\calK_{\bar t}$,
\begin{equation}\label{eq:sliced-field-bounds}
\|\bar g[\rho]\|_{L_\rho^2}^2
\leq
\int_{\calS^{n-1}}
\|g_\theta[\rho]\|_{L_\rho^2}^2
\,\sigma(\rmd\theta)
\leq
M_{\bar t}^2.
\end{equation}
\end{lemma}

\begin{proof}
See Appendix~\ref{sec:proof-lemma4}.
\end{proof}

It remains to verify the stability of the averaged drift and the local consistency of the averaged flow. For $X\in L^2(\Omega_0;\bbR^n)$, let $\bbP_X$ denote its law and define
\begin{equation}
b(t,X)
\defi
-\lambda(t)\bar g[\bbP_X](X).
\label{eq:b-def}
\end{equation}

\begin{lemma}[\bf Stability and local consistency]
\label{lem:euler-residual}
Suppose that Assumptions~\ref{ass:averaged-flow} and \ref{ass:main} hold. Then, for every $h\in(0,h_{\bar t}]$ and $k\in\{0,\ldots,K_h\}$,
\begin{equation}
\|b(t_k,x_k^h)-b(t_k,y(t_k))\|_{L^2(\Omega_0)}
\le
\Lambda_{\bar t}L_{\bar t}
\|x_k^h-y(t_k)\|_{L^2(\Omega_0)}
\label{eq:43-rewrite}
\end{equation}
for almost every realization of the direction history. Moreover, whenever $0\leq t\leq t+h\leq\bar t$, the averaged flow satisfies
\begin{equation}
y(t+h)
=
y(t)+hb(t,y(t))+r_{t,h},
\label{eq:44-rewrite}
\end{equation}
where
$\|r_{t,h}\|_{L^2(\Omega_0)}\leq B_{\bar t}h^2$ and
$B_{\bar t}
\defi
\frac{M_{\bar t}}{2}
\bigl(
\Lambda'_{\bar t}
+
\Lambda_{\bar t}^2L_{\bar t}
\bigr)$.
\end{lemma}

\begin{proof}
See Appendix~\ref{sec:proof-lemma5}.
\end{proof}

\subsection{Proof of Theorem~\ref{thm:main}}

We now assemble the preceding estimates. Let $\calH:=L^2(\Omega_0;\bbR^n)$ and $\|X\|_\calH^2:=\bbE_0[\|X\|^2]$, recall the drift $b$ defined in \eqref{eq:b-def}, and for $X\in\calH$ with $\bbP_X\in\calK_{\bar t}$, define $B(t,X,\theta):=- \lambda(t)g_\theta[\bbP_X](X)$. By Lemmas~\ref{lem:ac-preservation} and \ref{lem:moment-confinement}, whenever $h\in (0,h_{\bar{t}}]$, the random laws $\rho_k^{h}$ belong to $\calK_{\bar t}$ for all $k\le K_h$, almost surely with respect to the sampled directions. The averaged laws $\rho_{t_k}$ also belong to $\calK_{\bar t}$.

We first identify the random sliced iteration with the stochastic Euler scheme in Proposition~\ref{prop:hilbert-euler}. The recursion \eqref{eq:random-scheme} can be written as
$x_{k+1}^h
=
x_k^h+h b(t_k,x_k^h)+h\xi_{k+1}$,
where
$\xi_{k+1}
:=
B(t_k,x_k^h,\theta_k)-b(t_k,x_k^h)$. Since $\theta_k$ is independent of $\Theta_k$ and has distribution $\sigma$, $\bbE_{\{\theta_k\}}
\bigl[
\xi_{k+1}\mid\Theta_k
\bigr]
=
0$.
Thus, $\{\xi_{k+1}\}$ is a martingale difference sequence with respect to the filtration generated by the sampled directions.

It remains to bound its conditional second moment. We have
\begin{align*}
\|\xi_{k+1}\|_\calH^2
\le
\Lambda_{\bar t}^2
\left\|
g_{\theta_k}[\rho_k^{h}](x_k^h)
-
\bar g[\rho_k^{h}](x_k^h)
\right\|_\calH^2\le
2\Lambda_{\bar t}^2
\left\|
g_{\theta_k}[\rho_k^{h}](x_k^h)
\right\|_\calH^2
+
2\Lambda_{\bar t}^2
\left\|
\bar g[\rho_k^{h}](x_k^h)
\right\|_\calH^2.
\end{align*}
Taking conditional expectation and using
Lemma~\ref{lem:MS-bound} gives
\begin{align*}
\bbE_{\{\theta_k\}}
\! \left[
\|\xi_{k+1}\|_\calH^2
\mid\Theta_k
\right]\le
2\Lambda_{\bar t}^2
\bbE_{\{\theta_k\}}
\! \left[
\left\|
g_{\theta_k}[\rho_k^{h}](x_k^h)
\right\|_\calH^2
\mid\Theta_k
\right]
+
2\Lambda_{\bar t}^2
\left\|
\bar g[\rho_k^{h}](x_k^h)
\right\|_\calH^2\le
4\Lambda_{\bar t}^2M_{\bar t}^2.
\end{align*}
Hence, the martingale-noise condition in Proposition~\ref{prop:hilbert-euler} holds with $\sigma_\xi =2\Lambda_{\bar t}M_{\bar t}$.

We next verify the deterministic part of the Euler estimate. By Lemma~\ref{lem:euler-residual}, the trajectory-wise stability condition \eqref{eq:trajectory-stability-euler} holds for $z_k^h=x_k^h$ and $z(t_k)=y(t_k)$ with $L=\Lambda_{\bar t}L_{\bar t}$. Moreover, the trajectory under the averaged controller satisfies 
\[ 
y(t_{k+1}) = y(t_k)+h b(t_k,y(t_k))+r_k, \quad \|r_k\|_\calH \le B_{\bar t}h^2. 
\] 
Proposition~\ref{prop:hilbert-euler}, applied with $z_k^h=x_k^h$ and $z(t_k)=y(t_k)$, therefore yields
\begin{equation}
\scalebox{0.9}{$\bbE_{\{\theta_k\}}
\! \left[
\max_{0\le k\le K_h}
\|x_k^h-y(t_k)\|_{L^2(\Omega_0)}^2
\right]
\le
C(\bar t,L,B_{\bar t},\sigma_\xi)h$}.
\label{eq:H-main-bound}
\end{equation}

We now transfer this estimate from the lifted variables to their probability laws. For each fixed realization of the direction sequence, $x_k^h$ has law $\rho_k^{h}$ under $\bbP_0$, while $y(t_k)$ has law $\rho_{t_k}$. Hence, their joint law provides a coupling, and therefore
$$W_2^2
\bigl(
\rho_k^{h},
\rho_{t_k}
\bigr)
\le
\|x_k^h-y(t_k)\|_{L^2(\Omega_0)}^2.$$
Taking the maximum over $k$, then expectation with respect to the sampled directions, and using \eqref{eq:H-main-bound}, we obtain
\[
\bbE_{\{\theta_k\}}
\!\left[
\max_{0\le k\le K_h}
W_2^2
\bigl(
\rho_k^{h},
\rho_{t_k}
\bigr)
\right]
\le
C(\bar t,L,B_{\bar t},\sigma)h
\]
that proves \eqref{eq:main-W2}. The sliced Wasserstein estimate \eqref{eq:main-SW2} follows immediately from $SW_2^2(\rho,\eta)\le\frac{1}{n}W_2^2(\rho,\eta)$.

It remains to establish convergence at an arbitrary fixed time $t\in[0,\bar t]$. Let $k_h:=\lfloor t/h\rfloor$. Since $|t_{k_h}-t|\le h$, the temporal Lipschitz estimate \eqref{eq:Y-lipschitz-time} gives
\[
W_2(\rho_{t_{k_h}},\rho_t)
\le
\|y(t_{k_h})-y(t)\|_{L^2(\Omega_0)}
\le
\Lambda_{\bar t}M_{\bar t}h.
\]
By the triangle inequality,
\begin{align*}
W_2^2
\bigl(
\rho_{k_h}^{h},
\rho_t
\bigr)
&\le 
\left(
W_2
\bigl(
\rho_{k_h}^{h},
\rho_{t_{k_h}}
\bigr)
+
W_2(\rho_{t_{k_h}},\rho_t)
\right)^2\le 
2W_2^2
\bigl(
\rho_{k_h}^{h},
\rho_{t_{k_h}}
\bigr)
+
2W_2^2(\rho_{t_{k_h}},\rho_t)\\
&\le 
2\!\!\!\! \max_{0\le k\le K_h}\!\!\!
W_2^2
\bigl(
\rho_k^{h},
\rho_{t_k}
\bigr)
+
2W_2^2(\rho_{t_{k_h}},\rho_t).
\end{align*}
Taking expectations and using the preceding estimates, we obtain
\[
\bbE_{\{\theta_k\}}
\! \left[
W_2^2
\bigl(
\rho_{k_h}^{h},
\rho_t
\bigr)
\right]
\le
2C(\bar t,L,B_{\bar t},\sigma_\xi)h
+
2\Lambda_{\bar t}^2M_{\bar t}^2h^2.
\]
Consequently, $\lim_{h\searrow0} \bbE_{\{\theta_k\}}\! \big[W_2^2\bigl(\rho_{\lfloor t/h\rfloor}^{h},\rho_t\bigr)\big]=0$.

\section{Sliced OT Control for Fully Actuated Linear Systems}
\label{sec:full_actuated}

We now extend the sliced optimal transport controller from the single integrator to a linear time-varying system. The main difficulty is that the control input no longer prescribes the state velocity directly, as the motion is also shaped by the drift and the time-varying input directions. We therefore introduce coordinates that remove the drift and normalize the control authority over the steering horizon.

\subsection{Drift removal and reachability normalization}
\label{subsec:reachability_coordinates}

Consider the linear system
\begin{equation}\label{eq:ltv}
\dot{x}(t)
=
A(t)x(t)+B(t)u(t).
\end{equation}

\begin{assumption}\label{ass:full_actuation}
The matrices $A(t)$ and $B(t)$ are continuous on $[0,1]$, and $B(t)$ has full row rank for every $t\in[0,1]$.
\end{assumption}

Let $\Phi(t,s)$ denote the state-transition matrix associated with $A(\cdot)$. To remove the drift, set $F_t:=\Phi(1,t)$ and introduce $y(t)\defi F_t x(t)$. Since $\dot F_t=-F_tA(t)$, the transformed state satisfies $$\dot y(t)=F_tB(t)u(t).$$ The reachability Gramian over $[0,1]$ is 
\[
G_{10}
:=
\int_0^1
\Phi(1,s)B(s)B(s)^\top\Phi(1,s)^\top
\,\rmd s.
\]

Assumption~\ref{ass:full_actuation} implies that $G_{10}\succ0$. Let $L:=G_{10}^{-1/2}$ and introduce the reachability-normalized state $z(t)\defi Ly(t)=LF_t x(t).$ Its dynamics are $$\dot z(t)=C_tu(t),$$ where $C_t:=LF_tB(t)$. Defining $Q_t:=C_tC_t^\top$, we have $\int_0^1Q_t\,\rmd t=LG_{10}L^\top=I$. Thus, the transformed coordinates remove the drift and normalize the aggregate control authority over the steering horizon. Under Assumption~\ref{ass:full_actuation}, $C_t$ has full row rank and $Q_t\succ0$ for every $t\in[0,1]$. Since $Q_t$ is continuous, there exist constants $q_->0$ and $q_+<\infty$ such that 
\begin{equation}\label{eq:q_bounds}
q_-I
\preceq
Q_t
\preceq
q_+I
\end{equation}
for every $t\in[0,1]$.

\subsection{Feedback realization of the sliced flow}\label{subsec:feedback_realization}

We now construct a feedback law that makes the normalized state $z(t)=LF_t x(t)$ follow the averaged sliced flow of the single integrator. Let $\rho_t$ denote the law of $x(t)$. The corresponding law of $z(t)$ is $\what{\rho}_t\defi (LF_t)_\#\rho_t$. Since $F_1=I$, the target law in the normalized coordinates is $\what{\rho}_1=L_\#\rho_1$.

For each $\theta\in\calS^{n-1}$, let $\what{\calT}_t^\theta$ be the one-dimensional optimal transport map from $(P_\theta)_\#\what{\rho}_t$ to $(P_\theta)_\#\what{\rho}_1$. Define the associated sliced discrepancy field by
\[
\what{g}_t(z)
:=
\int_{\calS^{n-1}}
\! \left(
\theta^\top z
-
\what{\calT}_t^\theta(\theta^\top z)
\right)
\theta
\,\sigma(\rmd\theta).
\]

For a continuous gain $\lambda(t)$, the desired velocity in the normalized coordinates is $-\lambda(t)\what{g}_t(z)$. Since $\dot z(t)=C_tu(t)$, the control must satisfy 
$$C_tu(t)=-\lambda(t)\what{g}_t(z(t)).$$

Under Assumption~\ref{ass:full_actuation}, $C_t$ has full row rank and $Q_t=C_tC_t^\top$ is invertible. Thus, $C_t^\top Q_t^{-1}$ is a right inverse of $C_t$, and the minimum-norm control realizing the prescribed velocity is
\begin{equation}\label{eq:control_linsys}
u(t)
=
-\lambda(t)
C_t^\top Q_t^{-1}
\what{g}_t\bigl(LF_t x(t)\bigr).
\end{equation}

Substitution into the normalized dynamics gives $\dot z(t)=-\lambda(t)\what{g}_t(z(t))$. Hence, the normalized state follows exactly the averaged sliced flow developed for the single integrator, and its terminal convergence can be transferred to the original coordinates.

\begin{proposition}\label{prop:fully_actuated_convergence}
Suppose that Assumption~\ref{ass:full_actuation} holds and that the conditions of Theorem~\ref{thm:convergence} are satisfied for the transformed initial and target laws. Then, under the controller \eqref{eq:control_linsys}, $\lim_{t\nearrow1}m(t)=m_1$ and $\lim_{t\nearrow1}\Sigma(t)=\Sigma_1$.
\end{proposition}

\begin{proof}
Since the normalized state follows $\dot z(t) = -\lambda(t)\what g_t(z(t))$, the law of $z(t)$ follows the averaged sliced flow with target
$L_\#\rho_1$. 
Let $m_z(t)$ and $\Sigma_z(t)$ denote the mean and covariance of
$z(t)$. The transformed target law has mean $Lm_1$ and covariance
$L\Sigma_1L^\top$. Proposition~\ref{prop:linearlity_Gaussian} and
Theorem~\ref{thm:convergence} therefore give
$\lim_{t\nearrow1}m_z(t)
=
Lm_1$.
Likewise, we have 
$\lim_{t\nearrow1}\Sigma_z(t)
=
L\Sigma_1L^\top$.

The relation $z(t)=LF_t x(t)$ gives
$m_z(t)
=
LF_t m(t)$.
Since $L$ and $F_t$ are invertible,
$m(t)
=
F_t^{-1}L^{-1}m_z(t)$.
Continuity of the state-transition matrix and $F_1=I$ imply that
$\lim_{t\nearrow1}F_t^{-1}=I$. Taking the limit in the preceding
identity gives $\lim_{t\nearrow1}m(t)=m_1$. The covariances satisfy
$$\Sigma_z(t)
=
LF_t\Sigma(t)F_t^\top L^\top$$
Equivalently, it holds that
$$\Sigma(t)
=
F_t^{-1}L^{-1}
\Sigma_z(t)
L^{-\top}F_t^{-\top}$$
Taking the limit and using
$\lim_{t\nearrow1}F_t^{-1}=I$ gives
$\lim_{t\nearrow1}\Sigma(t)
=
\Sigma_1$.
\end{proof}

\subsection{Energy bounds}
\label{subsec:energy_bounds}

The preceding construction establishes terminal convergence through the transformed state $z(t)$. We now examine how the energy of this sliced flow is reflected in the control effort required by the original linear system. The relation between the two is governed by the matrices $Q_t$.

Define the sliced discrepancy in the transformed coordinates by $S(t)\defi SW_2^2\bigl(\what{\rho}_t,\what{\rho}_1\bigr)$. Equivalently, $S(t)=SW_2^2\bigl((LF_t)_\#\rho_t,L_\#\rho_1\bigr)$. As in Proposition~\ref{prop:time_derivative}, its derivative along the transformed density flow is
\begin{equation}\label{eq:S_derivative_general}
\dot S(t)
=
2\bbE
\! \left[
\what{g}_t(z(t))^\top\dot z(t)
\right].
\end{equation}
Under the controller \eqref{eq:control_linsys}, the transformed dynamics satisfy
\[
\dot z(t)
=
-\lambda(t)\what{g}_t(z(t)).
\]
Substitution into \eqref{eq:S_derivative_general} gives
$$\dot S(t)
=
-2\lambda(t)
\bbE
\! \left[
\|\what{g}_t(z(t))\|^2
\right].$$
Hence, $S(t)$ is non-increasing along the transformed density evolution.

We distinguish the physical input energy from the kinetic energy of the transformed flow defined by
\[
\calE_z(0,t)
\defi
\int_0^t
\bbE\! \left[\|\dot z(s)\|^2 \right]
\,\rmd s.
\]
Since $z(t)$ follows the averaged sliced flow, the following proposition compares the physical input energy
with the kinetic energy of the transformed flow.

\begin{proposition}[\bf Input-energy bounds]\label{thm:energy_identity}
Suppose that Assumption~\ref{ass:full_actuation} holds. Then, under the controller in \eqref{eq:control_linsys}, we have
that, for any $t\in[0,1]$,
\[
\frac{1}{q_+}\calE_z(0,t)
\le
\calE_u(0,t)
\le
\frac{1}{q_-}\calE_z(0,t).
\]
\end{proposition}
\vsp
\begin{proof}
Under the controller in \eqref{eq:control_linsys},
\begin{align}
\|u(t)\|^2
&=
\lambda(t)^2
\what{g}_t(z(t))^\top
Q_t^{-1}
\what{g}_t(z(t))=
\dot z(t)^\top
Q_t^{-1}
\dot z(t).
\label{eq:input_norm}
\end{align}

The bounds in \eqref{eq:q_bounds} imply
$\frac{1}{q_+}\|\xi\|^2
\leq
\xi^\top Q_t^{-1}\xi
\leq
\frac{1}{q_-}\|\xi\|^2$
for every $\xi\in\bbR^n$ and $t\in[0,1]$. Applying this inequality to $\xi=\dot z(t)$ and using \eqref{eq:input_norm} yields
$\frac{1}{q_+}\|\dot z(t)\|^2
\leq
\|u(t)\|^2
\leq
\frac{1}{q_-}\|\dot z(t)\|^2$.
Taking expectations and integrating over $[0,t]$ gives the desired bounds.
\end{proof}

Proposition~\ref{thm:energy_identity} shows that the physical control energy is equivalent, up to the constants $q_-$ and $q_+$, to the
kinetic energy of the transformed sliced flow. These constants quantify the distortion introduced when the averaged sliced flow is realized through time-varying input directions of the original system.

\section{Algorithm for General Controllable Systems}\label{sec:general_controllable}
The construction in Section~\ref{sec:full_actuated} relies on uniform full actuation. When $B(t)$ has full row rank, the matrix $Q_t$ is uniformly positive definite, and any prescribed velocity in the transformed coordinates can be realized instantaneously. The transformed dynamics can therefore be made to coincide with the averaged sliced flow of the single integrator, leading to terminal convergence and corresponding energy bounds.

For a general controllable system, however, $B(t)$ need not have full row rank. The matrix $Q_t$ may then be singular, and an arbitrary transformed velocity cannot, in general, be realized at each instant. Thus, the continuous-time feedback construction of Section~\ref{sec:full_actuated} is no longer directly applicable.

Controllability, nevertheless, ensures that a prescribed displacement can be realized over a finite time interval. This suggests replacing instantaneous velocity matching with finite-step displacement matching. Accordingly, over each interval of a finite partition, we first determine a sliced displacement and then construct a control input that realizes this displacement exactly by the end of the interval.

\subsection{Finite-step sliced update}

Consider a partition $0=t_0<t_1<\cdots<t_N=1$. We retain the fixed transformation introduced in Section~\ref{sec:full_actuated}. For each interval $J_k=[t_k,t_{k+1}]$, define the local lifted Gramian by
$\calG_k=\int_{t_k}^{t_{k+1}} C_sC_s^\top \rmd s$.

\begin{assumption}[\bf Local controllability]
\label{ass:local_controllability}
The system is controllable on each interval of the partition, i.e.,
the controllability Gramian satisfies
$\calG_k\succ0$ for every $k$.
\end{assumption}

Assumption~\ref{ass:local_controllability} requires the transformed system to be controllable over each interval $[t_k,t_{k+1}]$. It does not require $C_t$ to have full row rank at every time. Thus, although an arbitrary transformed velocity may not be realizable instantaneously, any prescribed finite displacement can be realized over the interval. For a time-invariant controllable pair $(A,B)$, the condition holds on every interval of positive length.

At time $t_k$, the transformed state and its law are $z_{t_k}=LF_{t_k}x(t_k)$
and $\widehat\rho_k=(LF_{t_k})_\#\rho_{t_k}$, respectively. The transformed target law remains fixed and is given by $\widehat\rho_1=L_\#\rho_1$.

Let $\Xi_k:\bbR^n\rightarrow\bbR^n$ denote a sliced update map constructed from $\widehat\rho_k$ and $\widehat\rho_1$. For a direction $\theta_k\in\calS^{n-1}$, one possible choice is
$$
\Xi_k(z)=z-\alpha_k\! \left(\theta_k^\top z - \widehat{\calT}_k^{\theta_k}(\theta_k^\top z) \right)\theta_k,
$$
where $\alpha_k$ is a step size and $\widehat{\calT}_k^{\theta_k}$ is the one-dimensional optimal transport map between the corresponding projected laws. Alternatively, using the averaged sliced discrepancy $\widehat g_k$, one may take $$\Xi_k(z)=z-\alpha_k\widehat g_k(z).$$

The desired displacement is $\delta_k(z)\defi \Xi_k(z)-z$. Given the sampled state $z_{t_k}$, we apply over the interval $J_k$ the input
\begin{equation}
u(t)
=
C_t^\top\calG_k^{-1}\delta_k(z_{t_k}).
\label{eq:lifted-input-z}
\end{equation}
Equivalently, 
$$u(t)
=
B(t)^\top F_t^\top L^\top
\calG_k^{-1}\delta_k(z_{t_k})$$ in original coordinates.
Here, $z_{t_k}=LF_{t_k}x(t_k)$, and the input is applied for $t\in J_k$.

\begin{proposition}[\bf Finite-step exactness]
\label{prop:finite_step_exact}
Suppose that Assumption~\ref{ass:local_controllability} holds. Under the input \eqref{eq:lifted-input-z}, the transformed state satisfies
$z_{t_{k+1}}=\Xi_k(z_{t_k})$. Consequently,
$\widehat\rho_{k+1}
=
(\Xi_k)_\#\widehat\rho_k$.
\end{proposition}

\begin{proof}
Over the interval $J_k$, the transformed dynamics satisfy $\dot z(t)=C_tu(t)$. Therefore,
\begin{align*}
z_{t_{k+1}}
=
z_{t_k} \!
+ \! \! \int_{t_k}^{t_{k+1}} \!\!\!\! C_su(s)\rmd s =
z_{t_k}
+
\int_{t_k}^{t_{k+1}} \!\!\!\! C_sC_s^\top\rmd s\,
\calG_k^{-1}\delta_k(z_{t_k})=
z_{t_k}+\delta_k(z_{t_k})
=
\Xi_k(z_{t_k}).
\end{align*}
The law update is thus $\widehat\rho_{k+1}=(\Xi_k)_\#\widehat\rho_k$.
\end{proof}

Proposition~\ref{prop:finite_step_exact} shows that, at the sampling instants, the transformed law follows exactly the virtual sliced iteration
$\widehat\rho_{k+1}=(\Xi_k)_\#\widehat\rho_k$.
Thus, the linear dynamics introduce no additional approximation at these instants. Any approximation arises solely from the choice of the sliced update $\Xi_k$.

\subsection{Algorithmic form}
The finite-step construction is summarized in Algorithm~\ref{alg:finite_step_sliced}.
\begin{algorithm}[htb!]
\caption{Finite-step sliced steering}
\label{alg:finite_step_sliced}
\begin{algorithmic}[1]
\Require Initial law $\rho_0$ and target law $\rho_1$
\Require Partition $0=t_0<t_1<\cdots<t_N=1$
\For{$k=0,\ldots,N-1$}
    \State Compute $z_{t_k}=LF_{t_k}x(t_k)$ and  $\widehat\rho_k=(LF_{t_k})_\#\rho_{t_k}$
    \State Construct sliced update map $\Xi_k$ from
    $\widehat\rho_k$ to $\widehat\rho_1$
    \State Compute   
    $\calG_k=\int_{t_k}^{t_{k+1}} LF_tB(t)B(t)^\top F_t^\top L^\top \rmd t$
    \State Apply, for $t\in[t_k,t_{k+1})$,
    $u(t)
    =
    B(t)^\top F_t^\top L^\top
    \calG_k^{-1}
    \left(
    \Xi_k(z_{t_k})-z_{t_k}
    \right)$
\EndFor
\end{algorithmic}
\end{algorithm}

This procedure is the finite-step counterpart of the iterative sliced controller for a general controllable linear system. At each sampling instant, projected transport maps are recomputed from the current transformed law, and the resulting sliced displacement is realized exactly over the following interval. The construction is intrinsically tied to a finite partition and should not be interpreted as a continuous-time extension of the fully actuated result. When $B(t)$ is rank deficient, the matrix $Q_t=C_tC_t^\top$ is singular. Although the local Gramian $\calG_k$ may remain positive definite on every interval, it can become increasingly ill-conditioned as the interval length decreases. Consequently, refining the partition does not, in general, produce a finite-energy continuous-time limit.

\section{Numerical Experiments}

All numerical experiments were implemented in MATLAB R2025b and performed on a MacBook equipped with an Apple M1 Pro chip and 16 GB of RAM. The experiments are designed to illustrate two aspects of the proposed framework: the distribution steering capability of the randomized sliced controller and its realization through linear dynamical systems.

\subsection{Example 1: Gaussian mixtures}

We first illustrate the randomized sliced controller on a non-Gaussian distribution-steering problem in $\bbR^2$. The initial and target laws are three-component Gaussian mixtures, $\rho_i=\sum_{j=1}^{3}w_{i,j}\mathcal N(m_{i,j},\Sigma_{i,j})$ for $i\in\{0,1\}$, with weights $w_0=(0.45,0.35,0.20)$ and $w_1=(0.40,0.30,0.30)$. The component means are given by the columns of
\begin{equation*}
M_0=
\begin{bmatrix}
-2.8 & -0.9 & -2.2\\
0.2 &  2.4 & -2.0
\end{bmatrix},
\quad
M_1=
\begin{bmatrix}
1.2 &  3.0 &  1.7\\
1.6 & -0.3 & -2.0
\end{bmatrix},
\end{equation*}
that is, $M_i=[m_{i,1},m_{i,2},m_{i,3}]$, where each $m_{i,j}\in\bbR^2$ is a column vector.

The corresponding covariance matrices are
\begin{equation*}
\begin{aligned}
\{\Sigma_{0,j}\}_{j=1}^{3}
&=
\left\{
\begin{bmatrix}
0.35 & 0.10\\
0.10 & 0.25
\end{bmatrix},
\begin{bmatrix}
0.20 & -0.06\\
-0.06 & 0.35
\end{bmatrix},
\begin{bmatrix}
0.22 & 0\\
0 & 0.18
\end{bmatrix}
\right\},\\
\{\Sigma_{1,j}\}_{j=1}^{3}
&=
\left\{
\begin{bmatrix}
0.25 & -0.05\\
-0.05 & 0.20
\end{bmatrix},
\begin{bmatrix}
0.30 & 0.09\\
0.09 & 0.25
\end{bmatrix},
\begin{bmatrix}
0.18 & 0\\
0 & 0.32
\end{bmatrix}
\right\}.
\end{aligned}
\end{equation*}

Both laws are represented by $N=10^4$ particles, and the target samples are drawn once and fixed throughout the experiment. At iteration $k$, a direction $\theta_k$ is sampled uniformly from $\mathcal S^1$. The empirical one-dimensional transport map is obtained by projecting the samples onto $\theta_k$, sorting the projected values, and matching them by rank. Using the remaining-horizon gain $\lambda(t)=(1-t)^{-1}$, we use the logarithmically refined grid $t_k=1-\Ee^{-0.008k}$ for $k=0,\ldots,2250$.

Figure~\ref{fig:ex1-traj} shows representative particle trajectories and density contours during the controlled evolution. The initial mixture is continuously deformed toward the target law, with mass originating from one component potentially redistributed among several target components.

\begin{figure}[htb!]
\centering
\includegraphics[width=0.75\linewidth,trim={1cm 0 1cm 0},clip]
{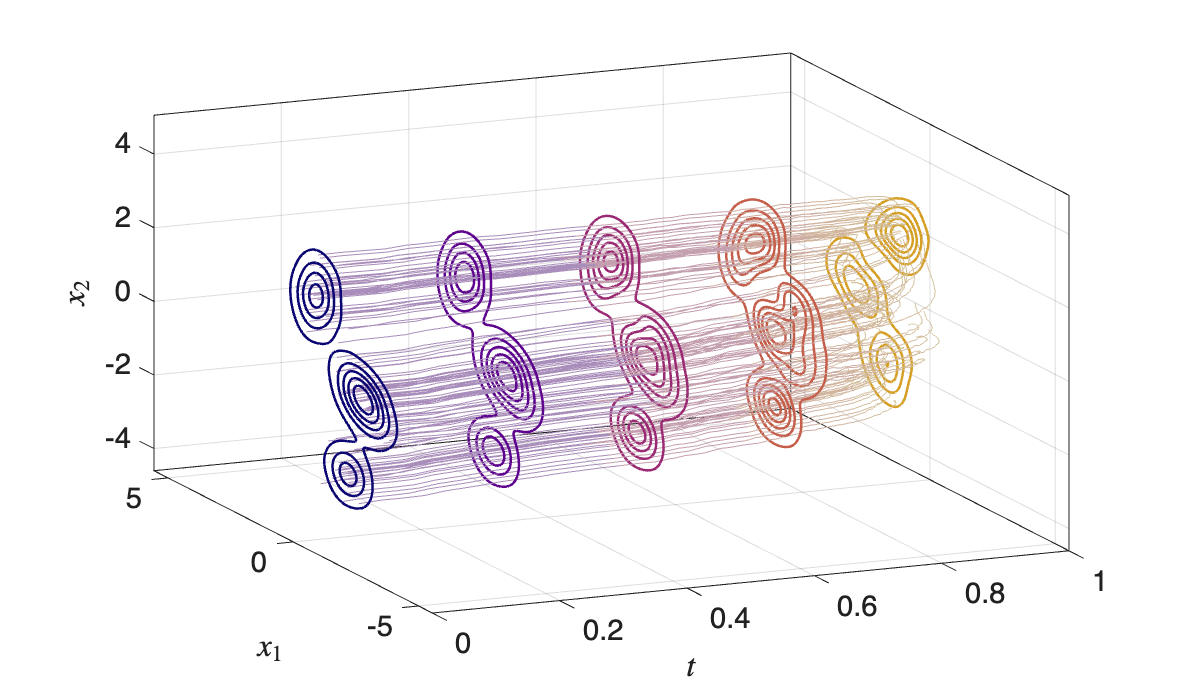}
\caption{Randomized sliced steering between two Gaussian mixtures.
Thin curves show $120$ representative particle trajectories, while
contours show the controlled density at
$t=0,0.25,0.5,0.75,1$. Mass originating from one initial component
may be distributed among several target components.}
\label{fig:ex1-traj}
\end{figure}

Figure~\ref{fig:ex1-sw} shows the empirical sliced Wasserstein distance to the target, estimated using $96$ projection directions. Its decay is consistent with the dissipation property in Proposition~\ref{prop:time_derivative}. Near the terminal time, the remaining discrepancy is comparable to the empirical discrepancy between two independent samples from the target law.

\begin{figure}[htb!]
\centering
\includegraphics[width=0.7\linewidth]
{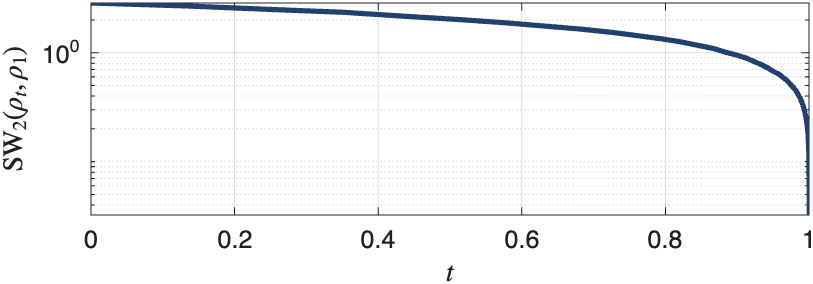}
\caption{Empirical sliced Wasserstein distance to the target over time,
estimated using $96$ projection directions. The final computed point
satisfies $1-t_{2250}\approx1.5\times10^{-8}$ and is identified with the terminal
time.}
\label{fig:ex1-sw}
\end{figure}

\subsection{Example 2: Color-distribution steering through fully actuated dynamics}
\label{subsec:color_transfer}

We next illustrate the realization of sliced distribution steering through fully actuated dynamics in a color-transfer problem. The initial and target laws are extracted from two oil paintings, a classical application of optimal transport in imaging (see, e.g., \cite{Peyre2019computational,solomon2015convolutional}). The initial image is Monet's \emph{Haystack}, while the target color distribution is extracted from \emph{Water Lilies}. This example demonstrates the fully actuated realization developed in Section~\ref{sec:full_actuated} and verifies the associated energy bounds.

Each pixel is represented by its CIELAB color vector $x=[L^\ast,a^\ast,b^\ast]^\top\in\bbR^3$. The spatial locations of pixels are fixed, and only the color distribution evolves. Thus, the problem concerns color-distribution steering rather than spatial deformation of the image. For numerical conditioning, we use the normalized coordinates $x= [(L^\ast-50)/50,\ a^\ast/110,\ b^\ast/110]^\top$. Both images are resized so that their largest dimension is $160$ pixels. All pixels of the initial image are used as particles, and the target quantiles are computed from the complete set of target-image pixels.

The color particles evolve according to the fully actuated linear system
\eqref{eq:ltv}, with
$$
A=
\begin{bmatrix}
-0.08 & -0.22 & 0.06\\
 0.18 & -0.06 & 0.03\\
 0.02 & -0.05 & -0.10
\end{bmatrix}$$
and
\[
B(t)=
\begin{bmatrix}
1.10+0.12\sin(2\pi t) & 0.05\cos(2\pi t) & 0\\
0.03\sin(2\pi t) & 0.95+0.10\cos(2\pi t) & 0.04\sin(4\pi t)\\
0.02\cos(2\pi t) & 0.03\sin(2\pi t) & 1.05+0.08\sin(4\pi t)
\end{bmatrix}.
\]
The matrix $B(t)$ has full row rank on $[0,1]$, and hence Assumption~\ref{ass:full_actuation} holds. The reachability Gramian $G_{10}$ is evaluated by the trapezoidal rule using $2001$ quadrature points. The matrices $L=G_{10}^{-1/2}$ and $Q_t$ are then computed as in Section~\ref{sec:full_actuated}, giving the uniform actuation bounds $q_-=0.7815$ and $q_+=1.2852$.

The interval $[0,1]$ is divided into $60$ uniform steps. At each step, the spherical average in the sliced controller is approximated using $32$ fixed directions on $\calS^2$, distributed approximately uniformly by a Fibonacci lattice. The projected optimal transport maps are obtained by sorting the projected color samples and matching empirical quantiles. The minimum-norm input $u=C_t^\top Q_t^{-1}\dot z$ then realizes the sliced velocity through the physical dynamics. Since the projection directions are fixed, the implementation is deterministic.

Figure~\ref{fig:color_setup} shows the two paintings and their empirical color distributions in the $(a^\ast,b^\ast)$ plane. The initial palette is concentrated around pale blue, gray, and ochre tones, whereas the target distribution contains stronger purple, green, and pink components. The controller moves the color distribution toward the target while preserving the spatial arrangement of the initial image.

\begin{figure}[htb!]
\centering
\includegraphics[width=.85\linewidth]{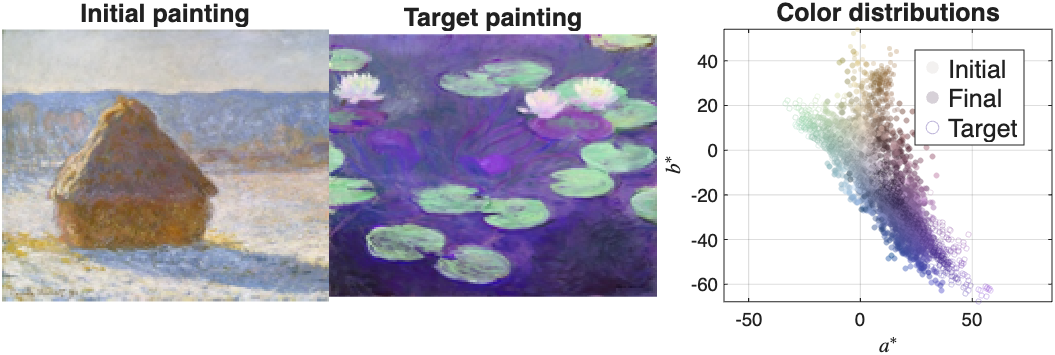}
\caption{Initial and target paintings and their empirical color distributions
in the $(a^\ast,b^\ast)$ plane. The initial painting provides the spatial
structure and initial color distribution, while the target painting provides
the desired color distribution.}
\label{fig:color_setup}
\end{figure}

Figure~\ref{fig:color_evolution} shows the resulting evolution at six equally spaced times. The pale palette of the initial painting is gradually replaced by the purple, green, and blue tones of the target. Since the controller acts only on color coordinates, the spatial structure remains unchanged. Over the horizon, the sliced distance to the target decreases from $0.299$ to $0.066$.

\begin{figure}[htb!]
\centering
\includegraphics[width=0.85\linewidth]{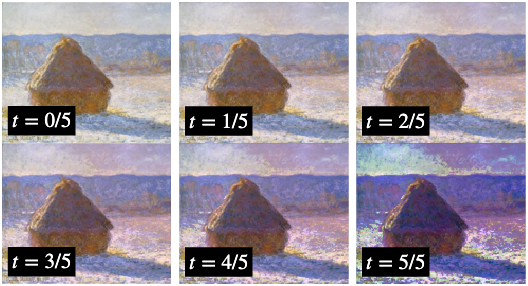}
\caption{Color-distribution steering through fully actuated linear dynamics.
Snapshots are shown at $t=0/5,1/5,\ldots,5/5$. The spatial structure of the
initial painting is preserved, while its color distribution is steered toward
that of the target painting.}
\label{fig:color_evolution}
\end{figure}

We next verify the energy characterization associated with the reachability normalization. Figure~\ref{fig:color_energy_bound} shows the accumulated input energy $\calE_u(0,t)$ together with the bounds obtained from the transformed energy $\calE_z(0,t)$. Throughout the horizon,
\[
\frac{\calE_z(0,t)}{q_+}
\leq
\calE_u(0,t)
\leq
\frac{\calE_z(0,t)}{q_-}.
\]
Thus, the physical energy required to realize the sliced flow through the
original dynamics is controlled by the actuation bounds in
\eqref{eq:q_bounds}.

\begin{figure}[htb!]
\centering
\includegraphics[width=0.85\linewidth]{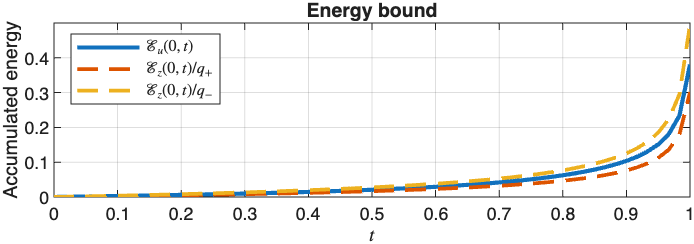}
\caption{Accumulated physical input energy and the corresponding lower and
upper bounds. The physical energy $\calE_u(0,t)$ remains between
$\calE_z(0,t)/q_+$ and $\calE_z(0,t)/q_-$ throughout the steering horizon.}
\label{fig:color_energy_bound}
\end{figure}

Finally, we compare the sliced controller with classical optimal transport as an energy benchmark. The exact optimal transport map between the transformed endpoint distributions gives the minimum-energy displacement interpolation. Since its computation requires solving an assignment problem over an $N\times N$ cost matrix, the comparison is performed on a common subsample of $N=1500$ pixels from each image. On this subsample, the physical control energy of the sliced controller is $1.31$ times the optimal transport minimum, indicating a moderate loss in energy optimality.

\begin{remark}[\bf Computational Complexity]\label{rem:computational_scaling}
Let $N$ denote the number of particles, $N_{\rm dir}$ the number of projection directions, and $K$ the number of time steps. A sliced update requires $O(N_{\rm dir}N\log N)$ operations and $O(N_{\rm dir}N)$ memory, leading to $O(KN_{\rm dir}N\log N)$ operations over the full horizon. In contrast, exact transport between equally weighted samples requires a dense $N\times N$ cost matrix and an assignment solve, with $\Theta(N^2)$ memory and cubic worst-case complexity.
\end{remark}

For Figure~\ref{fig:color_complexity}, we use $N_{\rm dir}=32$ and $K=60$. One exact map at $N=2000$ takes $15.3$ seconds, whereas the complete sliced run at the full image resolution $N=20{,}480$ takes $1.1$ seconds. At full resolution, the dense cost matrix alone requires approximately $3.4$ GB, compared with about $5.2$ MB for the projected arrays. These timings compare one exact endpoint map with a complete feedback run and illustrate their different scaling behavior.

\begin{figure}[htb!]
\centering
\includegraphics[width=0.85\linewidth]{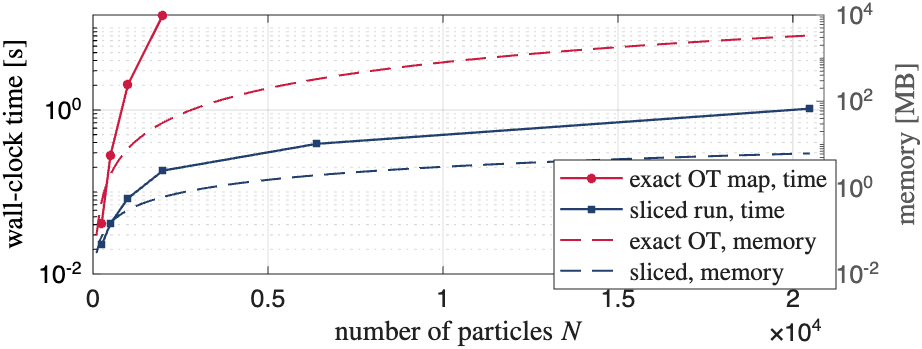}
\caption{Runtime and memory scaling of sliced and exact optimal transport. Solid lines show measured times for one exact endpoint map and a complete $60$-step sliced run. Dashed lines show the storage required by the dense $N\times N$ cost matrix and the $N_{\rm dir}\times N$ projected arrays. Exact transport becomes impractical beyond a few thousand particles, while the sliced controller extends to the full image resolution.}
\label{fig:color_complexity}
\end{figure}

\section{Conclusion}\label{sec:conclusion}
We developed a finite-horizon framework for distribution steering based on one-dimensional optimal transport along linear projections. A sampled projection direction defines a scalar endpoint-control problem, whose minimum-energy realization yields the iterative sliced controller. For the single-integrator dynamics, its averaged flow decreases the sliced Wasserstein discrepancy, admits a Gaussian specialization, and satisfies an explicit energy characterization. We further proved that the randomized controller converges to the averaged sliced flow as the sampling period vanishes. The construction was extended to linear dynamical systems. Reachability-normalized coordinates enable instantaneous realization of the sliced velocity for uniformly fully actuated systems, while local controllability provides exact finite-step realization for general controllable systems. These results show that projected transport maps can provide a tractable and dynamically realizable alternative to full-dimensional optimal transport. Future directions include adaptive projection selection, partial observations, and stochastic, constrained, or nonlinear dynamics.

\section*{Acknowledgment}

This work was supported in part by JSPS KAKENHI Grant Number
JP24K17297, JST ASPIRE Grant Number JPMJAP2402, the Swedish Research
Council Distinguished Professor Grant 2017-01078, and the Knut and Alice
Wallenberg Foundation Wallenberg Scholar Grant.

\bibliographystyle{alpha}
\bibliography{arXiv}

@Inproceedings{Thurin2026,
  title={Convergence Rates for Distribution Matching with Sliced Optimal Transport},
  author={Thurin, Gauthier and Boyer, Claire and Nadjahi, Kimia},
  booktitle={Proceedings of the 39th Annual Conference on Learning Theory},
  pages={6156--6196},
  year={2026},
  publisher={PMLR}
}

@article{Li2023measure,
  title={Measure transfer via stochastic slicing and matching},
  author={Li, Shiying and Moosmueller, Caroline and Wang, Yongzhe},
  journal={arXiv preprint arXiv:2307.05705},
  year={2023}
}

@inproceedings{Ito2026sliced,
  title={Sliced {W}asserstein steering between {G}aussian measures},
  author={Ito, Kaito and Dong, Anqi},
  booktitle={2026 European Control Conference (ECC)},
  year={2026},
  organization={IEEE},
}

@book{Teschl2012ordinary,
  title={Ordinary {D}ifferential {E}quations and {D}ynamical {S}ystems},
  author={Teschl, Gerald},
  year={2012},
  publisher={American Mathematical Soc.}
}

@book{Rachev1998mass,
  title={Mass {T}ransportation {P}roblems: {V}olume {I}: {T}heory},
  author={Rachev, Svetlozar T. and R{\"u}schendorf, Ludger},
  year={1998},
  publisher={Springer}
}

@book{Villani2003,
  title={{Topics in {O}ptimal {T}ransportation}},
  author={Villani, C{\'e}dric},
  volume={58},
  year={2003},
  publisher={American Mathematical Soc.}
}

@article{Peyre2019computational,
  title={Computational optimal transport: {W}ith applications to data science},
  author={Peyr{\'e}, Gabriel and Cuturi, Marco},
  journal={Foundations and Trends{\textregistered} in Machine Learning},
  volume={11},
  number={5-6},
  pages={355--607},
  year={2019},
  publisher={Now Publishers, Inc.}
}

@article{paulin2020sliced,
  title={Sliced optimal transport sampling},
  author={Paulin, Lois and Bonneel, Nicolas and Coeurjolly, David and Iehl, Jean-Claude and Webanck, Antoine and Desbrun, Mathieu and Ostromoukhov, Victor},
  journal={ACM Trans. Graph.},
  volume={39},
  number={4},
  pages={99},
  year={2020}
}

@article{Ito2023lcss,
  title={Entropic Model Predictive Optimal Transport for Underactuated Linear Systems},
  author={Ito, Kaito and Kashima, Kenji},
  journal={IEEE Control Systems Letters},
  volume={7},
  pages={2761--2766},
  year={2023},
  publisher={IEEE}
}

@article{Cozzi2025,
  title={Long-time asymptotics of the sliced-{W}asserstein flow},
  author={Cozzi, Giacomo and Santambrogio, Filippo},
  journal={SIAM Journal on Imaging Sciences},
  volume={18},
  number={1},
  pages={1--19},
  year={2025},
  publisher={SIAM}
}

@inproceedings{Liutkus2019,
  title={Sliced-{W}asserstein flows: {N}onparametric generative modeling via optimal transport and diffusions},
  author={Liutkus, Antoine and Simsekli, Umut and Majewski, Szymon and Durmus, Alain and St{\"o}ter, Fabian-Robert},
  booktitle={International Conference on Machine Learning},
  pages={4104--4113},
  year={2019},
  organization={PMLR}
}

@inproceedings{Vauthier2025,
title={Towards Understanding Gradient Dynamics of the Sliced-{W}asserstein Distance via Critical Point Analysis},
author={Christophe Vauthier and Anna Korba and Quentin M{\'e}rigot},
booktitle={Forty-second International Conference on Machine Learning},
year={2025},
}

@article{Ito2023entropic,
  title={Entropic model predictive optimal transport over dynamical systems},
  author={Ito, Kaito and Kashima, Kenji},
  journal={Automatica},
  volume={152},
  pages={110980},
  year={2023},
  publisher={Elsevier}
}

@article{Pitie2007automated,
  title={Automated colour grading using colour distribution transfer},
  author={Piti{\'e}, Fran{\c{c}}ois and Kokaram, Anil C. and Dahyot, Rozenn},
  journal={Computer Vision and Image Understanding},
  volume={107},
  number={1-2},
  pages={123--137},
  year={2007},
  publisher={Elsevier}
}

@article{Chen2017optimal,
  title={Optimal transport over a linear dynamical system},
  author={Chen, Yongxin and Georgiou, Tryphon T. and Pavon, Michele},
  journal={IEEE Transactions on Automatic Control},
  volume={62},
  number={5},
  pages={2137--2152},
  year={2017},
  publisher={IEEE}
}

@article{Benamou2000,
  title={{A computational fluid mechanics solution to the {M}onge-{K}antorovich mass transfer problem}},
  author={Benamou, Jean-David and Brenier, Yann},
  journal={Numerische Mathematik},
  volume={84},
  number={3},
  pages={375--393},
  year={2000},
  publisher={Springer-Verlag Berlin/Heidelberg}
}

@book{Ambrosio2008,
  title={{Gradient {F}lows: in {M}etric {S}paces and in the {S}pace of {P}robability {M}easures}},
  author={Ambrosio, Luigi and Gigli, Nicola and Savar{\'e}, Giuseppe},
  year={2008},
  edition={Second},
  publisher={Springer}
}

@article{Chen2021optimal,
  title={Optimal transport in systems and control},
  author={Chen, Yongxin and Georgiou, Tryphon T. and Pavon, Michele},
  journal={Annual Review of Control, Robotics, and Autonomous Systems},
  volume={4},
  number={1},
  pages={89--113},
  year={2021},
  publisher={Annual Reviews}
}

@article{Chen2018optimal,
  title={Optimal transport for {G}aussian mixture models},
  author={Chen, Yongxin and Georgiou, Tryphon T. and Tannenbaum, Allen},
  journal={IEEE Access},
  volume={7},
  pages={6269--6278},
  year={2019},
  publisher={IEEE}
}

@book{Deans2007radon,
  title={The {R}adon {T}ransform and {S}ome of {I}ts {A}pplications},
  author={Deans, Stanley R.},
  year={2007},
  publisher={Courier Corporation}
}

@article{Monge1781memoire,
  title={M{\'e}moire sur la th{\'e}orie des d{\'e}blais et des remblais},
  author={Monge, Gaspard},
  journal={Mem. Math. Phys. Acad. Royale Sci.},
  pages={666--704},
  year={1781}
}

@inproceedings{Kantorovich1942translocation,
  title={On the translocation of masses},
  author={Kantorovich, Leonid V.},
  booktitle={Dokl. Akad. Nauk. USSR (NS)},
  volume={37},
  pages={199--201},
  year={1942}
}

@article{dong2024monge,
  title={Monge--{K}antorovich optimal transport through constrictions and flow-rate constraints},
  author={Dong, Anqi and Stephanovitch, Arthur and Georgiou, Tryphon T.},
  journal={Automatica},
  volume={160},
  pages={111448},
  year={2024},
  publisher={Elsevier}
}

@inproceedings{nguyen2023sliced,
  title={Sliced {W}asserstein estimation with control variates},
  author={Nguyen, Khai and Ho, Nhat},
  booktitle={International Conference on Learning Representations},
  pages={2087--2106},
  year={2024}
}

@article{bonneel2015sliced,
  title={Sliced and {R}adon {W}asserstein barycenters of measures},
  author={Bonneel, Nicolas and Rabin, Julien and Peyr{\'e}, Gabriel and Pfister, Hanspeter},
  journal={Journal of Mathematical Imaging and Vision},
  volume={51},
  number={1},
  pages={22--45},
  year={2015},
  publisher={Springer}
}

@book{villani2009optimal,
  title={Optimal {T}ransport: {O}ld and {N}ew},
  author={Villani, C{\'e}dric},
  volume={338},
  year={2009},
  publisher={Springer}
}

@article{chen2015optimal1,
  title={Optimal steering of a linear stochastic system to a final probability distribution, {P}art {I}},
  author={Chen, Yongxin and Georgiou, Tryphon T. and Pavon, Michele},
  journal={IEEE Transactions on Automatic Control},
  volume={61},
  number={5},
  pages={1158--1169},
  year={2016},
  publisher={IEEE}
}

@article{chen2015optimal2,
  title={Optimal steering of a linear stochastic system to a final probability distribution, {P}art {II}},
  author={Chen, Yongxin and Georgiou, Tryphon T. and Pavon, Michele},
  journal={IEEE Transactions on Automatic Control},
  volume={61},
  number={5},
  pages={1170--1180},
  year={2016},
  publisher={IEEE}
}

@article{benamou2015iterative,
  title={Iterative {B}regman projections for regularized transportation problems},
  author={Benamou, Jean-David and Carlier, Guillaume and Cuturi, Marco and Nenna, Luca and Peyr{\'e}, Gabriel},
  journal={SIAM Journal on Scientific Computing},
  volume={37},
  number={2},
  pages={A1111--A1138},
  year={2015},
  publisher={SIAM}
}

@inproceedings{cuturi2014fast,
  title={Fast computation of {W}asserstein barycenters},
  author={Cuturi, Marco and Doucet, Arnaud},
  booktitle={International Conference on Machine Learning},
  pages={685--693},
  year={2014},
  organization={PMLR}
}

@article{stephanovitch2025optimal,
  title={Optimal transport through a toll station},
  author={Stephanovitch, Arthur and Dong, Anqi and Georgiou, Tryphon T.},
  journal={European Journal of Applied Mathematics},
  volume={36},
  number={3},
  pages={613--637},
  year={2025},
  publisher={Cambridge University Press}
}

@article{dong2025data,
  title={Data {A}ssimilation for {S}ign-indefinite {P}riors: {A} generalization of {S}inkhorn’s algorithm},
  author={Dong, Anqi and Georgiou, Tryphon T. and Tannenbaum, Allen},
  journal={Automatica},
  volume={177},
  pages={112283},
  year={2025},
  publisher={Elsevier}
}

@article{ito2024maximum,
  title={Maximum entropy density control of discrete-time linear systems with quadratic cost},
  author={Ito, Kaito and Kashima, Kenji},
  journal={IEEE Transactions on Automatic Control},
  volume={70},
  number={5},
  pages={3024--3039},
  year={2025},
  publisher={IEEE}
}

@article{solomon2015convolutional,
  title={Convolutional {W}asserstein distances: {E}fficient optimal transportation on geometric domains},
  author={Solomon, Justin and De Goes, Fernando and Peyr{\'e}, Gabriel and Cuturi, Marco and Butscher, Adrian and Nguyen, Andy and Du, Tao and Guibas, Leonidas},
  journal={ACM Transactions on Graphics (ToG)},
  volume={34},
  number={4},
  pages={1--11},
  year={2015},
  publisher={ACM New York, NY, USA}
}

@phdthesis{bonnotte2013unidimensional,
  title={Unidimensional and evolution methods for optimal transportation},
  author={Bonnotte, Nicolas},
  year={2013},
  school={Universit{\'e} Paris Sud-Paris XI; Scuola normale superiore (Pise, Italie)}
}

\newpage
\appendix
\section{Proofs in Section \ref{sec:single_integrator}}

\subsection{Proof of Proposition~\ref{prop:time_derivative}}
\label{app:time_derivative}

Fix $\bar t\in(0,1)$ and set, for almost every $t\in(0,\bar t)$,
\[
a(t)
:=
\left(
\int_{\bbR^n}
\|v(t,x)\|^2\,\rho_t(\rmd x)
\right)^{1/2}.
\]
By the continuity equation, the assumed integrability of $v$, and
\cite[Theorem~8.3.1]{Ambrosio2008}, the curve
$t\mapsto\rho_t$ is absolutely continuous in
$(\calP_2(\bbR^n),W_2)$ on $[0,\bar t]$, with
$|\rho_t'| := \lim_{s\rightarrow t} W_2 (\rho_s, \rho_t) /|s-t| \leq a(t)$ for almost every $t$. Hence,
\[
W_2(\rho_s,\rho_t)
\leq
\int_s^t a(\tau)\,\rmd\tau,
\qquad
0\leq s\leq t\leq\bar t.
\]
Using the triangle inequality for $SW_2$ and
$SW_2(\mu,\nu)\leq n^{-1/2}W_2(\mu,\nu)$, we obtain
$|r(t)-r(s)|
\leq
\frac{1}{\sqrt n}
\int_s^t a(\tau)\,\rmd\tau$.
Thus, $r(t)=SW_2(\rho_t,\rho_1)$ is absolutely continuous on
$[0,\bar t]$.

Fix $\theta\in\calS^{n-1}$ and define the projected velocity by
\[
w_{t,\theta}(y)
:=
\bbE\! \left[
\theta^\top v(t,X_t)
\,\middle|\,
\theta^\top X_t=y
\right],
\qquad
X_t\sim\rho_t.
\]
Jensen's inequality gives
\begin{equation}\label{eq:proj-velocity-short}
\int_{\bbR}
|w_{t,\theta}(y)|^2
\,\rho_{t,\theta}(\rmd y)
\leq
\int_{\bbR^n}
|\theta^\top v(t,x)|^2
\,\rho_t(\rmd x).
\end{equation}
For
$\varphi\in C_c^\infty((0,1)\times\bbR)$, substitute
$\zeta(t,x)=\varphi(t,\theta^\top x)$ into
\eqref{eq:dist_continuity}. This substitution is admissible by
\cite[Remark~8.1.1]{Ambrosio2008} and shows that
$(\rho_{t,\theta},w_{t,\theta})$ satisfies
$\partial_t\rho_{t,\theta}+\partial_y \! \left(w_{t,\theta}\rho_{t,\theta}\right)=0$ weakly.

It follows from \eqref{eq:proj-velocity-short} and
\cite[Theorem~8.3.1]{Ambrosio2008} that
$t\mapsto\rho_{t,\theta}$ is locally absolutely continuous in
$(\calP_2(\bbR),W_2)$. Since $\rho_t$ admits a density,
$\rho_{t,\theta}$ is absolutely continuous and hence atomless.
The optimal transport from $\rho_{t,\theta}$ to
$\rho_{1,\theta}$ is therefore induced by the monotone map
$\calT_t^\theta$. By
\cite[Theorem~8.4.7, Remark~8.4.8]{Ambrosio2008}, for almost every
$t\in(0,\bar t)$,
\begin{align}
\frac{\rmd}{\rmd t}
W_2^2(\rho_{t,\theta},\rho_{1,\theta})
=
2\int_{\bbR}
\left(
y-\calT_t^\theta(y)
\right)
w_{t,\theta}(y)
\,\rho_{t,\theta}(\rmd y)
=
2\int_{\bbR^n} \!
\left(
\theta^\top x-\calT_t^\theta(\theta^\top x)
\right)
\theta^\top v(t,x)
\,\rho_t(\rmd x).
\label{eq:projected-W2-derivative}
\end{align}

It remains to average this identity over the projection directions.
Since $t\mapsto\rho_t$ is continuous in $W_2$ on $[0,\bar t]$,
\[
C_{\bar t}
:=
\sup_{0\leq t\leq\bar t}
W_2(\rho_t,\rho_1)
<
\infty.
\]
Linear projections are nonexpansive, and hence
$W_2(\rho_{t,\theta},\rho_{1,\theta})\leq C_{\bar t}$.
By the Cauchy--Schwarz inequality, Eq.~\eqref{eq:projected-W2-derivative},
and
$\int_{\calS^{n-1}}|\theta^\top z|^2
\,\sigma(\rmd\theta)=\|z\|^2/n$,
\[
\int_{\calS^{n-1}}
\left|
\frac{\rmd}{\rmd t}
W_2^2(\rho_{t,\theta},\rho_{1,\theta})
\right|
\,\sigma(\rmd\theta)
\leq
\frac{2C_{\bar t}}{\sqrt n}\,a(t).
\]
The right-hand side is integrable on $(0,\bar t)$. Fubini's theorem
and the fundamental theorem of calculus therefore give, for almost
every $t\in(0,\bar t)$,
\[
\frac{\rmd}{\rmd t}
SW_2^2(\rho_t,\rho_1)
=
\int_{\calS^{n-1}}
\frac{\rmd}{\rmd t}
W_2^2(\rho_{t,\theta},\rho_{1,\theta})
\,\sigma(\rmd\theta).
\]
Substituting \eqref{eq:projected-W2-derivative} and using
\eqref{eq:sliced_discrepancy} yields
\[
\frac{\rmd}{\rmd t}
SW_2^2(\rho_t,\rho_1)
=
2\int_{\bbR^n}
g_t(x)^\top v(t,x)
\,\rho_t(\rmd x).
\]
Since $\bar t<1$ was arbitrary, the result holds locally on
$[0,1)$.

\subsection{Proof of Proposition~\ref{prop:linearlity_Gaussian}}\label{app:linearity_Gaussian}

We first establish that the covariance equation is well-posed on
$[0,1)$. Let $\bbS_{>0}^n$ denote the cone of real symmetric
positive-definite matrices. The vector field associated with
\eqref{eq:cov_eq} is
\[
F(t,\Sigma)
:=
\lambda(t)
\bigl(
H(\Sigma)\Sigma+\Sigma H(\Sigma)
\bigr).
\]
On every compact subset of $\bbS_{>0}^n$, the quantity
$\theta^\top\Sigma\theta$ is bounded uniformly away from zero. It follows
that $\alpha(\Sigma,\theta)$, and hence $H(\Sigma)$, is locally Lipschitz
in $\Sigma$, uniformly in $\theta$. Since $\lambda$ is continuous,
$F$ is locally Lipschitz in $\Sigma$, locally uniformly in time.
Therefore, Eq. \eqref{eq:cov_eq} along with $\Sigma(0)=\Sigma_0$ admits a unique
maximal solution on some interval $[0,t_{\max})$.

We next show that the solution cannot leave $\bbS_{>0}^n$ in finite time.
By the definition of $H$,
\[
H(\Sigma)+\frac{1}{n}I
=
\int_{\calS^{n-1}}
\alpha(\Sigma,\theta)\theta\theta^\top
\,\sigma(\rmd\theta)
\succeq0,
\]
and hence $H(\Sigma)\succeq-I/n$. Let
$\ell(t):=\lambda_{\min}(\Sigma(t))$. The function $\ell$ is locally
Lipschitz and therefore differentiable almost everywhere. At every such
time, choosing a unit eigenvector $q$ associated with $\ell(t)$ gives
\[
\begin{aligned}
\dot\ell(t)
=
q^\top\dot\Sigma(t)q=
2\lambda(t)\ell(t)
q^\top H(\Sigma(t))q\geq
-\frac{2}{n}\lambda(t)\ell(t).
\end{aligned}
\]
Gronwall's inequality yields
\begin{equation}\label{eq:sigma_positive}
\lambda_{\min}(\Sigma(t))
\geq
\lambda_{\min}(\Sigma_0)
\exp\left(
-\frac{2}{n}\int_0^t\lambda(s)\,\rmd s
\right)
>0.
\end{equation}

Suppose, by contradiction, that $t_{\max}<1$. Since $\lambda$ is
continuous, Eq.~\eqref{eq:sigma_positive} gives
$\Sigma(t)\succeq a_*I$ on $[0,t_{\max})$ for some $a_*>0$. Consequently,
\[
\alpha(\Sigma(t),\theta)
\leq
\sqrt{
\frac{\lambda_{\max}(\Sigma_1)}{a_*}
},
\qquad
t<t_{\max},
\]
and $H(\Sigma(t))$ is uniformly bounded on $[0,t_{\max})$. Thus, for some $c_*<\infty$ and $t<t_{\max}$, we have 
$$\|\dot\Sigma(t)\|_2 \leq 2c_*\lambda(t)\|\Sigma(t)\|_2.$$
Another application of Gronwall's inequality shows that
$\|\Sigma(t)\|_2\leq b_*$ on $[0,t_{\max})$ for some $b_*<\infty$.

The solution therefore remains in the compact subset
$$
\left\{
\Sigma\in\bbS_{>0}^n
\ \middle|\
a_*I\preceq\Sigma\preceq b_*I
\right\}.$$
This contradicts the continuation theorem for ordinary differential
equations \cite[Section~2.6]{Teschl2012ordinary}. Hence,
Eq.~\eqref{eq:cov_eq} admits a unique solution satisfying
$\Sigma(t)\succ0$ throughout $[0,1)$.

Since $\Sigma$ and $\lambda$ are continuous, so are
$K(t,\Sigma(t))$ and $\eta(t)$. The affine equation
\eqref{eq:X_affine}, with initial law
$X(0)\sim\calN(m_0,\Sigma_0)$, therefore admits a unique solution on
$[0,1)$. The variation-of-constants formula shows that $X(t)$ remains
Gaussian. Its mean satisfies
\begin{equation}\label{eq:m_ode}
\dot m(t)
=
-\frac{\lambda(t)}{n}
\bigl(m(t)-m_1\bigr),
\end{equation}
with $m(0)=m_0$, while its covariance satisfies \eqref{eq:cov_eq} with
initial value $\Sigma_0$. The solution of \eqref{eq:m_ode} is
\eqref{eq:m_solution}. Uniqueness of the mean and covariance equations
therefore gives
\[
X(t)\sim\calN\bigl(m(t),\Sigma(t)\bigr),
\qquad
t\in[0,1).
\]

It remains to identify the affine drift with the averaged sliced feedback.
For the Gaussian law
$\rho_t=\calN(\,\cdot\mid m(t),\Sigma(t))$, the projected transport map is
\[
\calT_t^\theta(\theta^\top x)
=
\theta^\top m_1
+
\alpha(\Sigma(t),\theta)
\theta^\top\bigl(x-m(t)\bigr).
\]
Substitution into Eq.~\eqref{eq:sliced_discrepancy}, together with the
definitions of $H$, $K$, and $\eta$, gives
\[
-\lambda(t)g_t(x)
=
K(t,\Sigma(t))x+\eta(t).
\]
Thus, Eq.~\eqref{eq:X_affine} coincides with the averaged sliced feedback.

\subsection{A Gaussian coercivity estimate}
\label{app:harmonic}

For $\Sigma\succ0$, define
\begin{equation}\label{eq:V_def}
V(\Sigma)
:=
\int_{\calS^{n-1}}
\! \left(
\sqrt{\theta^\top\Sigma\theta}
-
\sqrt{\theta^\top\Sigma_1\theta}
\right)^2
\,\sigma(\rmd\theta).
\end{equation}

\begin{lemma}[\bf Gaussian coercivity and directional consistency]
\label{lem:chi_gaussian}
Let
$\rho_t=\calN(\,\cdot\mid m(t),\Sigma(t))$
be a Gaussian averaged sliced flow generated by a nonnegative gain $\lambda$. Set
\[
\beta
:=
n \Big(
\sqrt{\frac{\trace(\Sigma_1)}{n}}
+
r(0)
\Big)^2
\]
and
\begin{equation}\label{eq:mu_def}
\mu
:=
\frac{
\sqrt{\lambda_{\min}(\Sigma_1)}
}{
\beta^{3/2}n(n+2)
}.
\end{equation}
Then, for every $t\in[0,1)$, we have 
\begin{equation}\label{eq:V-strong-lower}
\Sigma(t)\preceq\beta I \; \mbox{ and }\;
V(\Sigma(t))
\geq
\frac{\mu}{2}
\|\Sigma(t)-\Sigma_1\|_{\rm F}^2.
\end{equation}
If, in addition, $\Sigma(t)\succeq\alpha I$ on $[0,1)$ for some
$\alpha>0$, then  $\chi_*\leq\chi(t)\leq1$
whenever $r(t)>0$, where $\chi_*$ is the constant in
Theorem~\ref{thm:convergence}.
\end{lemma}

\begin{proof}
For the Gaussian law $\rho_t$, the sliced discrepancy field takes the
form $g_t(x)
=
\frac{1}{n}\bigl(m(t)-m_1\bigr)
-
H(\Sigma(t))\bigl(x-m(t)\bigr).$
Consequently,
\begin{align}
D(t)
&=
\frac{1}{n^2}\|m(t)-m_1\|^2
+
\trace\! \left(
H(\Sigma(t))\Sigma(t)H(\Sigma(t))^\top
\right),
\label{eq:gaussian-D}\\
r(t)^2
&=
\frac{1}{n}\|m(t)-m_1\|^2
+
V(\Sigma(t)).
\label{eq:gaussian-r}
\end{align}

We first bound the covariance. By
Eq.~\eqref{eq:dissipation_general}, $r(t)\leq r(0)$, and hence
$V(\Sigma(t))\leq r(0)^2$. On the other hand, the Cauchy--Schwarz
inequality gives
\[
V(\Sigma)
\geq
\Big(
\sqrt{\frac{\trace(\Sigma)}{n}}
-
\sqrt{\frac{\trace(\Sigma_1)}{n}}
\Big)^2.
\]
It follows that $\trace(\Sigma(t))\leq\beta$, and therefore
$\Sigma(t)\preceq\beta I$.

We next establish the coercivity of $V$. For a symmetric matrix
$\Delta$, the first and second directional derivatives of $V$ in the direction $\Delta$ are $D_\Delta V(\Sigma)
= -\trace\bigl(H(\Sigma)\Delta\bigr)$ and 
\begin{align}
D_\Delta^2V(\Sigma)
=
\int_{\calS^{n-1}}
\frac{
\sqrt{\theta^\top\Sigma_1\theta}
}{
2(\theta^\top\Sigma\theta)^{3/2}
}
\bigl(\theta^\top\Delta\theta\bigr)^2
\,\sigma(\rmd\theta).
\label{eq:V-hessian}
\end{align}
On the convex set
$\{\Sigma\succ0\mid\Sigma\preceq\beta I\}$, the coefficient in
\eqref{eq:V-hessian} is bounded below by
$\sqrt{\lambda_{\min}(\Sigma_1)}/(2\beta^{3/2})$. Moreover,
\[
\int_{\calS^{n-1}}
(\theta^\top\Delta\theta)^2
\,\sigma(\rmd\theta)
=
\frac{
(\trace\Delta)^2+2\|\Delta\|_{\rm F}^2
}{
n(n+2)
} \geq
\frac{2}{n(n+2)}\|\Delta\|_{\rm F}^2.
\]
Thus, $V$ is $\mu$-strongly convex on this set. Since
$V(\Sigma_1)=0$ and $\nabla V(\Sigma_1)=0$, strong convexity gives
the second inequality in \eqref{eq:V-strong-lower}. It also implies
\begin{equation}\label{eq:PL-V}
\|H(\Sigma)\|_{\rm F}^2
=
\|\nabla V(\Sigma)\|_{\rm F}^2
\geq
2\mu V(\Sigma).
\end{equation}

Suppose now that $\Sigma(t)\succeq\alpha I$. From
Eq.~\eqref{eq:PL-V}, we have
\[
\trace\! \left(
H(\Sigma(t))\Sigma(t)H(\Sigma(t))^\top
\right)
\geq
2\alpha\mu V(\Sigma(t)).
\]
Set $a:=\|m(t)-m_1\|^2/n$. Using
\eqref{eq:gaussian-D} and \eqref{eq:gaussian-r}, we obtain
\[
\begin{aligned}
\chi(t)
&=
\frac{
a+
n\trace\! \left(
H(\Sigma(t))\Sigma(t)H(\Sigma(t))^\top
\right)
}{
a+V(\Sigma(t))
}\geq
\frac{
a+
\displaystyle
\frac{
2\alpha\sqrt{\lambda_{\min}(\Sigma_1)}
}{
\beta^{3/2}(n+2)
}
V(\Sigma(t))
}{
a+V(\Sigma(t))
}
\geq
\chi_*.
\end{aligned}
\]
The upper bound $\chi(t)\leq1$ follows from
\eqref{eq:chi-upper}.
\end{proof}

\subsection{Proof of Theorem~\ref{thm:convergence}}
\label{app:convergence}

Set $f(t):=r(t)^2$. By \eqref{eq:dissipation_general} and
\eqref{eq:r-chi-def}, for almost every $t\in(0,1)$ such that $f(t)>0$,
\[
\dot f(t)
=
-2\lambda(t)D(t)
=
-\frac{2\lambda(t)\chi(t)}{n}f(t).
\]
Lemma~\ref{lem:chi_gaussian} gives $\chi(t)\geq\chi_*$. Hence,
$\dot f(t)
\leq
-\frac{2\chi_*}{n}\lambda(t)f(t)$.
If $f$ reaches zero, its monotonicity makes the desired estimate
immediate. Otherwise, Gronwall's inequality gives
\[
f(t)
\leq
f(0)
\exp\! \left(
-\frac{2\chi_*}{n}
\int_0^t\lambda(s)\,\rmd s
\right).
\]
Taking square roots yields \eqref{eq:r-decay-general}. In view of
\eqref{eq:cumulative-gain}, we also have $r(t)\to0$ as $t\nearrow1$. Finally, Eq.~\eqref{eq:gaussian-r} gives
\[
\frac{1}{n}\|m(t)-m_1\|^2
\leq
r(t)^2,
\qquad
V(\Sigma(t))
\leq
r(t)^2.
\]
Therefore, $m(t)\to m_1$. Moreover, by \eqref{eq:V-strong-lower}, it holds that 
\[
\frac{\mu}{2}
\|\Sigma(t)-\Sigma_1\|_{\rm F}^2
\leq
V(\Sigma(t))
\leq
r(t)^2,
\]
and hence $\Sigma(t)\to\Sigma_1$ as $t\nearrow1$.

\subsection{Proof of Lemma~\ref{lem:energy_characterization}}
\label{app:energy_lem}

By Proposition~\ref{prop:time_derivative} and
\eqref{eq:dissipation_general}, $r$ is locally absolutely continuous and
$\frac{\rmd}{\rmd t}r(t)^2
=
-2\lambda(t)D(t)$
for almost every $t\in(0,1)$. Whenever $r(t)>0$, it follows that
\begin{equation}\label{eq:r-dot-general}
\dot r(t)
=
-\lambda(t)\frac{D(t)}{r(t)}.
\end{equation}
On the other hand, since
$u(t)=-\lambda(t)g_t(X(t))$ with $X(t)\sim\rho_t$, we have 
$$\bbE\! \left[\|u(t)\|^2\right]
=
\lambda(t)^2D(t).$$ Using $D(t)=\chi(t)r(t)^2/n$ in \eqref{eq:r-dot-general} gives
$$\displaystyle \bbE\! \left[\|u(t)\|^2\right]
=
\frac{n\dot r(t)^2}{\chi(t)}$$ whenever $r(t)>0$.

If $r(t_0)=0$ for some $t_0<1$, then the nonnegativity and
monotonicity of $r^2$ imply that $r(t)=0$ for all $t\geq t_0$.
Moreover, Eq.~\eqref{eq:chi-upper} gives $D(t)=0$ there, and hence both
$\dot r(t)$ and $\bbE[\|u(t)\|^2]$ vanish for almost every
$t\geq t_0$. With the convention $\chi(t)=1$ when $r(t)=0$,
Eq.~\eqref{eq:power-chi} therefore holds for almost every $t\in(0,1)$.
Integration over $(0,\tau)$ gives \eqref{eq:energy-chi}.

Suppose now that $r(t)\to0$ as $t\nearrow1$. Since
$0<\chi(t)\leq1$ whenever $r(t)>0$, and both sides vanish after $r$
reaches zero, Eq.~\eqref{eq:energy-chi} yields
$\calE_u(0,\tau)\geq n\int_0^\tau \dot r(t)^2\,\rmd t$ for every
$\tau\in(0,1)$.
By the Cauchy--Schwarz inequality,
\[
\int_0^\tau\dot r(t)^2\,\rmd t
\geq
\frac{1}{\tau}\! \left(\int_0^\tau\dot r(t)\,\rmd t\right)^2
=
\frac{(r(\tau)-r(0))^2}{\tau}.
\]
Letting $\tau\nearrow1$ gives
$\calE_u(0,1)
\geq
nr(0)^2
=
nSW_2^2(\rho_0,\rho_1)$,
which proves \eqref{eq:energy-lower}.

\subsection{Proof of Theorem~\ref{thm:sw_gain}}
\label{app:energy}

Let
$\tau_*
:=
\inf\{t\in[0,1):r(t)=0\}$ and 
$\inf\varnothing:=1$.
Since $r(0)>0$, we have $r(t)>0$ on $[0,\tau_*)$. Thus,
$\lambda_{\rm SW}$ is well defined there, and
\eqref{eq:r-dot-general} gives, for almost every $t\in(0,\tau_*)$,
\[
\dot r(t)
=
-\lambda_{\rm SW}(t)\frac{D(t)}{r(t)}
=
-\frac{r(t)}{1-t}.
\]
Hence, we have $\displaystyle \frac{\rmd}{\rmd t}\left(\frac{r(t)}{1-t}\right)=0$
for almost every $t\in(0,\tau_*)$, and therefore $r(t)=(1-t)r(0)$ for $t\in[0,\tau_*)$.
If $\tau_*<1$, continuity of $r$ would give $r(\tau_*)=0$, whereas the preceding identity yields $r(\tau_*)=(1-\tau_*)r(0)>0$. Thus, $\tau_*=1$, and \eqref{eq:r-linear} holds on $[0,1)$. Since $\dot r(t)=-r(0)$ for almost every $t\in(0,1)$,
Lemma~\ref{lem:energy_characterization} gives for almost every $ t \in (0,1) $,
\[
\bbE \! \left[ \| u(t)\|^2 \right]=\frac{n\dot r(t)^2}{\chi(t)}
=\frac{nr(0)^2}{\chi(t)} .
\]
Integrating over $[0,1)$ and using
$r(0)=SW_2(\rho_0,\rho_1)$ proves \eqref{eq:energy-sw}.

Finally, $\chi(t)\leq1$ gives $\calE_u(0,1)
\geq
nSW_2^2(\rho_0,\rho_1)$.
If $\chi(t)\geq\underline{\chi}>0$ almost everywhere, then
$1/\chi(t)\leq1/\underline{\chi}$, which gives the upper bound in
\eqref{eq:energy-sw-bounds}.

\subsection{Proof of Corollary~\ref{cor:Gaussian_SW_normalized}}
\label{app:cor_gauss}
We first show that $D(t)>0$ whenever $r(t)>0$. From
\eqref{eq:gaussian-D}, if $D(t)=0$, then
$m(t)=m_1$ and
$\trace\left(
H(\Sigma(t))\Sigma(t)H(\Sigma(t))^\top
\right)=0$.
Since $\Sigma(t)\succ0$, this implies $H(\Sigma(t))=0$. As shown in
the proof of Lemma~\ref{lem:chi_gaussian}, the function $V$ is strictly
convex and satisfies $\nabla V=-H$. Hence,
$H(\Sigma(t))=0$ implies $\Sigma(t)=\Sigma_1$. Equation
\eqref{eq:gaussian-r} then gives $r(t)=0$. Therefore,
$D(t)>0$ whenever $r(t)>0$.

Theorem~\ref{thm:sw_gain} now applies and yields
$r(t)=(1-t)r(0)$. By \eqref{eq:gaussian-r}, we have
$$\|m(t)-m_1\|
\leq
\sqrt{n}\,r(t)
=
\sqrt{n}(1-t)r(0),$$
which proves \eqref{eq:m_bound}. Moreover,
$V(\Sigma(t))\leq r(t)^2$, and
\eqref{eq:V-strong-lower} gives
\[
\|\Sigma(t)-\Sigma_1\|_{\rm F}
\leq
\sqrt{\frac{2}{\mu}}\,r(t)
=
\sqrt{\frac{2}{\mu}}\,(1-t)r(0).
\]
This proves \eqref{eq:Sigma_bound} and, in particular,
$\Sigma(t)\to\Sigma_1$ as $t\nearrow1$.

Define $\overline{\Sigma}(t)=\Sigma(t)$ for $t<1$ and
$\overline{\Sigma}(1)=\Sigma_1$. The map
$\overline{\Sigma}:[0,1]\to\bbS_{>0}^n$ is continuous. Therefore,
\[
\alpha
:=
\min_{t\in[0,1]}
\lambda_{\min}\bigl(\overline{\Sigma}(t)\bigr)
>
0,
\]
and $\Sigma(t)\succeq\alpha I$ on $[0,1)$. Lemma
\ref{lem:chi_gaussian} then gives $\chi(t)\geq\chi_*$ whenever
$r(t)>0$. The energy bounds follow from
Theorem~\ref{thm:sw_gain}.

\section{Proofs in Section \ref{sec:convergence}}

\subsection{Proof of Lemma~\ref{lem:slice-id}}
\label{sec:proof-lemma1}

By the definition of $g_\theta[\rho]$ and pushforward identity
$\rho_\theta=(P_\theta)_\#\rho$,
\begin{align*}
\|g_\theta[\rho]\|_{L_\rho^2}^2
=
\int_{\bbR^n}
\left|
\theta^\top x-\calT_\rho^\theta(\theta^\top x)
\right|^2
\,\rmd\rho(x)=
\int_{\bbR}
\left|
s-\calT_\rho^\theta(s)
\right|^2
\,\rmd\rho_\theta(s)
=
W_2^2(\rho_\theta,\rho_{1,\theta}),
\end{align*}
where the last equality follows from the optimality of
$\calT_\rho^\theta$.

For the second claim, the triangle inequality gives
\[
\begin{aligned}
W_2(\rho_\theta,\rho_{1,\theta})
\leq
W_2(\rho_\theta,\delta_0)
+
W_2(\delta_0,\rho_{1,\theta})
=
m_2(\rho_\theta)+m_2(\rho_{1,\theta})\leq
m_2(\rho)+m_2(\rho_1),
\end{aligned}
\]
where $\delta_0$ denotes the Dirac measure at $0$ and the last inequality follows from
$|\theta^\top x|\leq\|x\|$.

\subsection{Proof of Lemma~\ref{lem:ac-preservation}}
\label{sec:proof-lemma2}

Fix $\theta\in\calS^{n-1}$ and write
$x=s\theta+z$, with $s\in\bbR$ and $z\in\theta^\perp \defi \{ z \in \bbR^n  \mid  \theta^\top z = 0\}$. In these
coordinates, the map in \eqref{eq:iter_map} takes the form
\[
T_{\rho,\theta,\gamma}(s,z)
=
\bigl(\phi(s),z\bigr),
\quad
\phi(s)
:=
(1-\gamma)s+\gamma\calT_\rho^\theta(s).
\]
Since $\calT_\rho^\theta$ is nondecreasing, we have 
$$\phi(s_2)-\phi(s_1)
\geq
(1-\gamma)(s_2-s_1)$$
whenever $s_2>s_1$. Thus, $\phi$ is strictly increasing, and its
inverse on $\phi(\bbR)$ is $(1-\gamma)^{-1}$-Lipschitz. In particular,
\begin{equation}\label{eq:null_phi}
\mu_L^1\bigl(\phi^{-1}(N)\bigr)
=
0
\end{equation}
for every Lebesgue-null set $N\subset\bbR$.

Let $N\subset\bbR^n$ satisfy $\mu_L^n(N)=0$, and denote its section
at $z\in\theta^\perp$ by
$N_z
:=
\{s\in\bbR:(s,z)\in N\}$.
By Fubini's theorem, $\mu_L^1(N_z)=0$ for almost every
$z\in\theta^\perp$. For each such $z$, the corresponding section of
the inverse image is
$\bigl(T_{\rho,\theta,\gamma}^{-1}(N)\bigr)_z
=
\phi^{-1}(N_z)$,
which is null by \eqref{eq:null_phi}. Another application of Fubini's
theorem gives
$\mu_L^n\bigl(T_{\rho,\theta,\gamma}^{-1}(N)\bigr)
=
0$.
Since $\rho\ll\mu_L^n$,
$(T_{\rho,\theta,\gamma})_\#\rho(N)
=
\rho\bigl(T_{\rho,\theta,\gamma}^{-1}(N)\bigr)
=
0$,
and hence
$(T_{\rho,\theta,\gamma})_\#\rho\ll\mu_L^n$.

It remains to verify the second-moment condition. By Minkowski's
inequality and Lemma~\ref{lem:slice-id},
\[
\begin{aligned}
m_2\bigl((T_{\rho,\theta,\gamma})_\#\rho\bigr)
=
\|T_{\rho,\theta,\gamma}\|_{L_\rho^2}
\leq
m_2(\rho)
+
\gamma\|g_\theta[\rho]\|_{L_\rho^2}
\leq
m_2(\rho)
+
\gamma\bigl(m_2(\rho)+m_2(\rho_1)\bigr)
<
\infty.
\end{aligned}
\]
Therefore,
$(T_{\rho,\theta,\gamma})_\#\rho
\in\calP_{2,\mathrm{ac}}(\bbR^n)$.

Finally, the update \eqref{eq:law-update_iter} has this form with
$\gamma=h\lambda(t_k)$. Since
$h\lambda(t_k)\leq h_{\bar t}\Lambda_{\bar t}<1$, the conclusion follows
recursively from
$\rho_0\in\calP_{2,\mathrm{ac}}(\bbR^n)$.

\subsection{Proof of Lemma~\ref{lem:moment-confinement}}
\label{sec:proof-lemma3}

Set $\gamma_k:=h\lambda(t_k)$. Conditionally on the direction history,
Minkowski's inequality and Lemma~\ref{lem:slice-id} give
\[
\begin{aligned}
m_2(\rho_{k+1}^h)
\leq
m_2(\rho_k^h)
+
\gamma_k
\|g_{\theta_k}[\rho_k^h]\|_{L_{\rho_k^h}^2} \leq
(1+\gamma_k)m_2(\rho_k^h)
+
\gamma_k m_2(\rho_1).
\end{aligned}
\]
Writing $a_k:=m_2(\rho_k^h)$ and $c:=m_2(\rho_1)$, we obtain
\[
a_{k+1}+c
\leq
(1+\gamma_k)(a_k+c).
\]
Iteration yields
\[
\begin{aligned}
a_k+c
\leq
(a_0+c)
\prod_{i=0}^{k-1}(1+\gamma_i)\leq
(a_0+c)
\exp\! \left(
h\sum_{i=0}^{k-1}\lambda(t_i)
\right)\leq
\bigl(m_2(\rho_0)+m_2(\rho_1)\bigr)
\Ee^{\Lambda_{\bar t}\bar t}.
\end{aligned}
\]
Dropping $c$ from the left-hand side proves
\eqref{eq:random-moment-bound}.

For the averaged flow, Minkowski's inequality and
Lemma~\ref{lem:slice-id} give
\[
\|\bar g[\rho_t]\|_{L_{\rho_t}^2}
\leq
\int_{\calS^{n-1}}
\|g_\theta[\rho_t]\|_{L_{\rho_t}^2}
\,\sigma(\rmd\theta)
\leq
m_2(\rho_t)+m_2(\rho_1).
\]
Using the integral form of \eqref{eq:ideal-flow}, we therefore have
\[
m_2(\rho_t)
\leq
m_2(\rho_0)
+
\int_0^t
\lambda(s)
\bigl(
m_2(\rho_s)+m_2(\rho_1)
\bigr)
\,\rmd s.
\]
Gronwall's inequality gives
\[
\begin{aligned}
m_2(\rho_t)
\leq
\bigl(m_2(\rho_0)+m_2(\rho_1)\bigr)
\exp\! \left(
\int_0^t\lambda(s)\,\rmd s
\right)\leq
\bigl(m_2(\rho_0)+m_2(\rho_1)\bigr)
\Ee^{\Lambda_{\bar t}\bar t}.
\end{aligned}
\]
Thus, it follows from Lemma~\ref{lem:ac-preservation} that both iterative and averaged laws remain in
$\calK_{\bar t}$ on $[0,\bar t]$.

\subsection{Proof of Lemma~\ref{lem:MS-bound}}
\label{sec:proof-lemma4}

By Lemma~\ref{lem:slice-id} and the definition of
$\calK_{\bar t}$,
$\|g_\theta[\rho]\|_{L_\rho^2}
\leq m_2(\rho)+m_2(\rho_1)\leq M_{\bar t}$.
Since $\sigma$ is a probability measure,
\[
\int_{\calS^{n-1}}
\|g_\theta[\rho]\|_{L_\rho^2}^2
\,\sigma(\rmd\theta)
\leq
M_{\bar t}^2.
\]
Moreover, Jensen's inequality and Fubini's theorem give
\[
\begin{aligned}
\|\bar g[\rho]\|_{L_\rho^2}^2
=
\int_{\bbR^n}
\left\|
\int_{\calS^{n-1}}
g_\theta[\rho](x)\,\sigma(\rmd\theta)
\right\|^2
\rmd\rho(x)\leq
\int_{\calS^{n-1}}
\|g_\theta[\rho]\|_{L_\rho^2}^2
\,\sigma(\rmd\theta).
\end{aligned}
\]
Combining the two estimates proves \eqref{eq:sliced-field-bounds}.

\subsection{Proof of Lemma~\ref{lem:euler-residual}}
\label{sec:proof-lemma5}

By Assumption~\ref{ass:main}, we have
\begin{align*}
\|b(t_k,x_k^h)-b(t_k,y(t_k))\|_{L^2(\Omega_0)}&=\lambda(t_k)
\left\|
\bar g[\rho_k^{h}](x_k^h)
-
\bar g[\rho_{t_k}](y(t_k))
\right\|_{L^2(\Omega_0)}\\
&\le
\Lambda_{\bar t}L_{\bar t}
\|x_k^h-y(t_k)\|_{L^2(\Omega_0)}.
\end{align*}
This proves \eqref{eq:43-rewrite}. We next estimate the local truncation error. By
Lemmas~\ref{lem:moment-confinement} and~\ref{lem:MS-bound}, it holds that 
$\|b(t,y(t))\|_{L^2(\Omega_0)}
\leq
\Lambda_{\bar t}M_{\bar t}$ for $t\in[0,\bar t]$.
The integral form of \eqref{eq:ideal-flow} therefore gives
\begin{equation}\label{eq:Y-lipschitz-time}
\|y(s)-y(t)\|_{L^2(\Omega_0)}
\leq
\Lambda_{\bar t}M_{\bar t}|s-t|
\end{equation}
for $s,t\in[0,\bar t]$. 

The definition of $b$,
Assumption~\ref{ass:main}, and
\eqref{eq:Y-lipschitz-time} yield
\[
\begin{aligned}
&\|b(s,y(s))-b(t,y(t))\|_{L^2(\Omega_0)}\\
&\leq
|\lambda(s)-\lambda(t)|
\|\bar g[\rho_s](y(s))\|_{L^2(\Omega_0)}
+
\lambda(t)
\|\bar g[\rho_s](y(s))
-
\bar g[\rho_t](y(t))\|_{L^2(\Omega_0)}\\
&\leq
M_{\bar t}
\bigl(
\Lambda'_{\bar t}
+
\Lambda_{\bar t}^2L_{\bar t}
\bigr)
|s-t|.
\end{aligned}
\]

Finally, define
$r_{t,h}
:=
\int_t^{t+h}
\bigl(
b(s,y(s))-b(t,y(t))
\bigr)
\,\rmd s$.
The integral form of the averaged dynamics gives
\eqref{eq:44-rewrite}, while
\[
\begin{aligned}
\|r_{t,h}\|_{L^2(\Omega_0)}
&\leq
M_{\bar t}
\bigl(
\Lambda'_{\bar t}
+
\Lambda_{\bar t}^2L_{\bar t}
\bigr)
\int_t^{t+h}|s-t|\,\rmd s=
B_{\bar t}h^2,
\end{aligned}
\]
where
$B_{\bar t}
\defi
\frac{M_{\bar t}}{2}
\bigl(
\Lambda'_{\bar t}
+
\Lambda_{\bar t}^2L_{\bar t}
\bigr)$.

\end{document}